\documentclass[11pt,reqno,oneside]{amsart}

\usepackage[utf8]{inputenc}
\usepackage[english]{babel}

\usepackage[a4paper, total={6in, 9in}]{geometry}

\usepackage{pgf,tikz,pgfplots} \pgfplotsset{compat=newest}
\usetikzlibrary{arrows}
\usetikzlibrary[patterns]
\usepackage{upgreek}
\usepackage{amsmath,amscd, amssymb, amsthm, mathrsfs}
\usepackage[abbrev,nobysame,alphabetic]{amsrefs}
\usepackage{xcolor}
\usepackage{cases}
\usepackage{hyperref}

\newtheorem{Theorem}{Theorem}[section]
\newtheorem*{Theorem*}{Theorem}
\newtheorem{Lemma}[Theorem]{Lemma}
\newtheorem{Corollary}[Theorem]{Corollary}
\newtheorem{Proposition}[Theorem]{Proposition}

\theoremstyle{definition}

\newtheorem{Assumption}{Assumption}[section]
\newtheorem{Remark}{Remark}[section]
\newtheorem*{Remark*}{Remark}

\numberwithin{equation}{section}

\newcommand{\R}{\mathbb{R}}
\newcommand{\Rn}{{\mathbb{R}^n}}

\usepackage[pagewise]{lineno} % Line numbers for our draft version
\newcommand*\patchAmsMathEnvironmentForLineno[1]{%
	\expandafter\let\csname old#1\expandafter\endcsname\csname #1\endcsname
	\expandafter\let\csname oldend#1\expandafter\endcsname\csname end#1\endcsname
	\renewenvironment{#1}%
	{\linenomath\csname old#1\endcsname}%
	{\csname oldend#1\endcsname\endlinenomath}}% 
\newcommand*\patchBothAmsMathEnvironmentsForLineno[1]{%
	\patchAmsMathEnvironmentForLineno{#1}%
	\patchAmsMathEnvironmentForLineno{#1*}}%
\AtBeginDocument{%
	\patchBothAmsMathEnvironmentsForLineno{equation}%
	\patchBothAmsMathEnvironmentsForLineno{align}%
	\patchBothAmsMathEnvironmentsForLineno{flalign}%
	\patchBothAmsMathEnvironmentsForLineno{alignat}%
	\patchBothAmsMathEnvironmentsForLineno{gather}%
	\patchBothAmsMathEnvironmentsForLineno{multline}%
	\patchAmsMathEnvironmentForLineno{subequations}%
	\patchAmsMathEnvironmentForLineno{numcases}%
}
\DeclareMathOperator{\divr}{div}

\let\Re\relax
\DeclareMathOperator{\Re}{\mathrm{Re}}

\newcommand{\dif}[1]{\,\mathrm{d}{#1}}

\newcommand{\df}{\mathrm{d}}
\newcommand{\nrm}[2][]{ \| {#2} \|_{#1}}
\newcommand{\agl}[1][\cdot]{ \langle {#1} \rangle}

\allowdisplaybreaks

\title[Non-timelike regularity lower order]{Regularity and energy of hyperbolic boundary value problems on non-timelike hypersurfaces with lower order terms}

\author{Shiqi Ma}
\address{School of Mathematics, Jilin University, Changchun, 130012, China}
\email{mashiqi@jlu.edu.cn, mashiqi01@gmail.com}

\begin{document}

\begin{abstract}
	
	We study second order hyperbolic equations with initial conditions, a nonhomogeneous Dirichlet boundary condition and a source term. We prove the solution possesses $H^1$ regularity on any piecewise $C^1$-smooth non-timelike hypersurfaces. We generalize the notion of energy to these hypersurfaces, and establish an estimate of the difference between square roots of energies on these hypersurfaces and on the initial plane where the time $t = 0$. The energy is shown to be conserved when the source term and the boundary datum are both zero. We also obtain an $L^2$ estimate for the normal derivative of the solution. We establish these results for $C^2$-smooth solutions first by using multiplier methods, then we go back to the original setting using approximation.

	\medskip

	
	
\end{abstract}

\maketitle

\section{Introduction} \label{sec:Intro-HR21}

Throughout the article we assume that $T$ is a fixed positive constant, $\Omega$ is an open domain in $\Rn~(n \geq 1)$ with $C^\infty$-smooth boundary.
$\Omega$ can be chosen to be unbounded or to be the whole space $\Rn$.
$A(x) = (a_{ij}(x)) \in C^{n+4}(\overline \Omega; \R^{n^2})$ is a real-valued symmetric $n \times n$ matrix function, and $q \in C_c^\infty(\overline \Omega;[0,+\infty))$ is a compactly supported potential with non-negative value. 
We focus on the following second order hyperbolic initial/boundary value problem,
\begin{equation} \label{eq:1-HR21}
	\left\{\begin{aligned}
	\partial_t^2 u - \nabla \cdot (A(x) \nabla u) + qu & = G && \text{in} \ Q:= \Omega \times (0,T), \\
	u & = f && \text{on} \ \Sigma := \partial \Omega \times (0,T), \\
	u = u_0,\,
	u_t & = u_1 && \text{on} \ \Omega \times \{t = 0\},
	\end{aligned}\right.
\end{equation}
under the condition
\begin{equation} \label{eq:1con-HR21}
	\left\{\begin{aligned}
	& G \in L^2(Q), \ f \in H^1(\Sigma) := L^2(0,T; H^1(\partial \Omega)) \cap H^1(0,T; L^2(\partial \Omega)), \\
	& u_0 \in H^1(\Omega), \ u_1 \in L^2(\Omega), \ u_0(x) = f(x,0) \ \text{on} \ \partial \Omega, \quad q \in C_c^{\lceil n/2 \rceil + 3}(\overline \Omega;[0,+\infty)).
	\end{aligned}\right.
\end{equation}
For the well-posedness of the solution, we also impose the following ellipticity condition on $A(x)$, namely, there exist two positive constants $c_1$ and $c_2$ such that for all $(x,\xi) \in \overline \Omega \times \mathbb C^n$,
\begin{equation} \label{eq:ellp-HR21}
	c_1 |\xi|^2 \leq |\xi|_A^2 := \overline \xi^T A(x) \xi \leq c_2 |\xi|^2.
\end{equation}

Existence and regularity results for the solution of \eqref{eq:1-HR21} under variant initial/boundary data can be found in the literature.
When $q = 0$, the existence of the unique solution $u \in C([0,T]; H^1(\Omega))$ of \eqref{eq:1-HR21} under the condition \eqref{eq:1con-HR21} is given in \cite[(3.5)]{llt1986non}.
In \cite{Las1981cos} the authors proved another existence result for $f \in L^2(0,T; L^2(\partial \Omega))$ and $G = 0$ by using the cosine operators technique, and they also showed the map from $f$ to $u$ is continuous.
Then, in \cite{Las1983reg} the same authors improved the regularity of $u$ from $L^2(0,T; L^2(\Omega))$ to $C([0,T]; L^2(\Omega))$.
In \cites{Saka1970mixedI, Saka1970mixedII} Sakamoto studied the problem \eqref{eq:1-HR21} with higher order regularities by using pseudo-differential operators.
The book \cite{Lions1972non} contains a comprehensive treatments of non-homogeneous boundary value problems, including hyperbolic equations.
See also \cites{Ava97sharp, las88lift} for related applications.

These studies mentioned above treat the solution $u$ as maps $[0,T] \to X$ where $X$ are function spaces defined in $\Omega$,
namely, the cylinder $Q$ are foliated horizontally and $u$ are defined in every horizontal slice $\Omega \times \{t = \tau\}$ for all $\tau \in [0,T]$.
In this article, we investigate properties of the solution of \eqref{eq:1-HR21}-\eqref{eq:1con-HR21} on \emph{non-timelike} hypersurfaces $\Gamma_S$ described by
\[
\Gamma_S := \{ (x, S(x)) \,;\, x \in \overline \Omega \}.
\]
Check Fig.~\ref{fig:geo0-HR21} as an example.
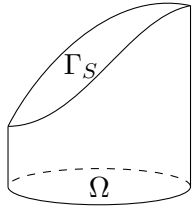
\begin{figure}[h]
	\tikzset{
		partial ellipse/.style args={#1:#2:#3}{
			insert path={+ (#1:#3) arc (#1:#2:#3)}
		}
	}
	\centering
	\begin{tikzpicture}[scale = 0.8, baseline=(current bounding box.center)]
		\pgfmathsetmacro{\LONG}{1.5}; % long radius of the ellipse
		\pgfmathsetmacro{\SHORT}{0.3}; % short radius of the ellipse
		\pgfmathsetmacro{\HEIGHT}{1.5}; % hieght of the cylinder
		\pgfmathsetmacro{\LP}{1}; % height of the left end of the hypersurface
		\pgfmathsetmacro{\RP}{2}; % height of the right end of the hypersurface
		\coordinate (BL) at (-\LONG,0);
		\coordinate (BR) at (\LONG,0);
		\coordinate (TL) at (-\LONG,\LP);
		\coordinate (TR) at (\LONG,\RP);
		\draw[white] (0,\HEIGHT) ellipse ({\LONG} and {\SHORT});
		\draw[dashed] (\LONG,0) arc(0:180:{\LONG} and {\SHORT});
		\draw (-\LONG,0) arc(180:360:{\LONG} and {\SHORT});
		\draw (BL) -- (TL);
		\draw (BR) -- (TR);
		\draw (-\LONG,\LP) .. controls (-0.2*\LONG,0.7*\LP + 0.1*\RP) and (0.3*\LONG,0.1*\LP + 0.9*\RP) .. ({\LONG},\RP);
		\draw (-\LONG,\LP) .. controls (-0.7*\LONG,0.7*\LP + 0.4*\RP) and (0.1*\LONG,1.1*\RP) .. ({\LONG},\RP);
		\node at (-0.2*\LONG,1.0*\HEIGHT) {$\Gamma_S$};
		\node at (0,0) {$\Omega$};
	\end{tikzpicture}
	\caption{An example of the non-timelike hypersurfaces $\Gamma_S$.}
	\label{fig:geo0-HR21}
\end{figure}

When regarding $\Gamma_S$ as a submanifold of $Q$, the upward normal vector $\nu$ of $\Gamma_S$ is given by
\begin{equation} \label{eq:nu-HR21}
	\nu := (\nu_x, \nu_t) = \Big( \frac {-\nabla S} {\sqrt{1+|\nabla S|^2}}, \frac 1 {\sqrt{1+|\nabla S|^2}} \Big).
\end{equation}
We say a hypersurface is timelike (resp.~spacelike, lightlike) with respect to $A$ at a given point $p$ if and only if the normal vector $\nu$ at $p$ is spacelike (resp.~timelike, lightlike), namely, $p$ satisfies
\[
|\nu_x|_A > |\nu_t|, ~(\text{resp.}~ |\nu_x|_A < |\nu_t|, \ |\nu_x|_A = |\nu_t|).
\]
And the hypersurface is said to be timelike (resp.~lightlike, spacelike) if it is timelike (resp.~lightlike, spacelike) at every point.
In this article we restrict our attention only to non-timelike (i.e.~lightlike or spacelike) hypersurfaces.
Therefore, we put the following assumption.

\begin{Assumption} \label{asp:nS-HR21}
	$S \in C^1(\overline \Omega)$ piecewisely with $S(x) \in [0,T]$, and $\Gamma_S$ is a non-timelike hypersurface, i.e.~$|\nabla S(x)|_A \leq 1$ for all $x \in \overline \Omega$.
\end{Assumption}

Restrictions of the solution on these slanted hypersurfaces have already appeared in the literature.
In \cite{RakeshUhlmann}, the authors studied the equation $(\partial_t^2 - \Delta + q) U = 0$ incited by an incident wave $\delta(t - x_1)$.
When encountered with the potential $q$, the incident wave generates a scattered wave $u$ such that
\[
U(x,t) = \delta(t - x_1) + u(x,t) H(t - x_1)
\]
where $H$ is the Heaviside function.
\cite[Theorem 1]{RakeshUhlmann} proves that the scattered wave $u$ is also a solution of the equation $(\partial_t^2 - \Delta + q) u = 0$ with the following initial condition
\[
u(t,x',t) = -\frac 1 2 \int_{-\infty}^t q(s, x') \dif s, \quad \forall (x',t) \in \Rn.
\]
The expression above for $u(t,x',t)$ involves the restriction of $u$ on the lightlike hypersurface $\{ (t,x', t) \}$ in $\R^{n+1}$.
Similar situations also appeared in \cites{RakeshSalo2, RakeshSalo1, merono2020fixed}.

\subsection{Main results}

In this work we establish $H^1$ regularity and energy estimates for the solution of \eqref{eq:1-HR21}-\eqref{eq:1con-HR21} on any non-timelike hypersurfaces.
The energy is estimated not directly on the value of the energy itself, but on the \emph{difference between square roots of the energies} on the hypersurface and the energy at time $t = 0$.
The estimate of the difference is sharper than the estimate of the value of the energy itself, see Theorem \ref{thm:1H1-HR21} and the discussion afterwards for details.

To state the main results, we introduce several notations.
We define the energy $\mathcal E(u; \Gamma_S)$ of $u$ on $\Gamma_S$ whenever the following expression can be well-defined:
{\small \begin{equation} \label{eq:en-HR21}
	\mathcal E(u; \Gamma_S)
	:= \frac 1 2 \int_{\Omega} \big[ |\nabla \big( u(x, S(x)) \big)|_A^2 + (1 - |\nabla S(x)|_A^2) |u_t(x, S(x))|^2 + q(x)|u(x,S(x))|^2 \big] \dif x.
\end{equation}}
When $S(x) = \text{constant}$, \eqref{eq:en-HR21} coincides with the classical energy definition.
We simplify the energy at time $t=0$ as $\mathcal E(u,0)$,
\begin{equation} \label{eq:E0-HR21}
	\mathcal E(u,0) := \mathcal E(u; S(x) = 0) = \frac 1 2 \int_\Omega (|\nabla u_0|_A^2 + |u_1|^2 + q(x)|u_0|^2) \dif x.
\end{equation}
The first result involves \eqref{eq:1-HR21} underneath $\Gamma_S$, so let us introduce the following notations (see Fig.~\ref{fig:geo-HR21}):
\begin{equation*}
	\left\{\begin{aligned}
		& Q_\tau := \{(x,t) \,;\, x \in \Omega,\, \tau \leq t \leq S(x) \} \subset Q, &&
		\Sigma_\tau := \{(x,t) \,;\, x \in \partial \Omega,\, \tau \leq t \leq S(x) \} \subset \Sigma, \\
		& H_\tau := Q_0 \cap \{(x,\tau)\,;\, x \in \Omega \}, &&
		\Gamma_{S,\tau} := \{ (x, S(x)) \,;\, \tau \leq S(x) \leq T_2 \} \\
		& T_1 := \inf \{ S(x) \,;\, x \in \Omega\}, &&
		T_2 := \sup \{ S(x) \,;\, x \in \Omega\}.
	\end{aligned}\right.
\end{equation*}
Note that when $\tau \leq T_1$, $\Gamma_{S,\tau} = \Gamma_S$.
And when $\Gamma_S$ is not horizontal, $Q_0$ and $\Sigma_0$ will be strict subsets of $Q$ and $\Sigma$, respectively.
\begin{figure}[h]
	\tikzset{
		partial ellipse/.style args={#1:#2:#3}{
			insert path={+ (#1:#3) arc (#1:#2:#3)}
		}
	}
	\centering
	\begin{tikzpicture}[scale = 0.8, baseline=(current bounding box.center)]
		\pgfmathsetmacro{\LONG}{1.5}; % long radius of the ellipse
		\pgfmathsetmacro{\SHORT}{0.3}; % short radius of the ellipse
		\pgfmathsetmacro{\HEIGHT}{4}; % hieght of the cylinder
		\pgfmathsetmacro{\LP}{1}; % height of the left end of the hypersurface
		\pgfmathsetmacro{\RP}{3}; % height of the right end of the hypersurface
		\coordinate (BL) at (-\LONG,0);
		\coordinate (BR) at (\LONG,0);
		\coordinate (TL) at (-\LONG,\HEIGHT);
		\coordinate (TR) at (\LONG,\HEIGHT);
		\draw (0,\HEIGHT) ellipse ({\LONG} and {\SHORT});
		\draw[dashed] (\LONG,0) arc(0:180:{\LONG} and {\SHORT});
		\draw (-\LONG,0) arc(180:360:{\LONG} and {\SHORT});
		\draw (BL) -- (TL);
		\draw (BR) -- (TR);
		\node[anchor=north west] at (-\SHORT, 0.8*\HEIGHT) {$Q$};
		\node[anchor=north west] at (\LONG, 0.8*\HEIGHT) {$\Sigma$};
		\node at (0,0) {$\Omega$};
	\end{tikzpicture}
	\quad \quad
	\begin{tikzpicture}[scale = 0.8, baseline=(current bounding box.center)]
		\pgfmathsetmacro{\LONG}{1.5}; % long radius of the ellipse
		\pgfmathsetmacro{\SHORT}{0.3}; % short radius of the ellipse
		\pgfmathsetmacro{\HEIGHT}{4}; % hieght of the cylinder
		\pgfmathsetmacro{\LP}{1}; % height of the left end of the hypersurface
		\pgfmathsetmacro{\RP}{3}; % height of the right end of the hypersurface
		\coordinate (BL) at (-\LONG,0);
		\coordinate (BR) at (\LONG,0);
		\coordinate (TL) at (-\LONG,\LP);
		\coordinate (TR) at (\LONG,\RP);
		\draw[white] (0,\HEIGHT) ellipse ({\LONG} and {\SHORT});
		\draw[dashed] (\LONG,0) arc(0:180:{\LONG} and {\SHORT});
		\draw (-\LONG,0) arc(180:360:{\LONG} and {\SHORT});
		\draw[dashed] (0,\LP*1.6) [partial ellipse=0:130:{\LONG} and {\SHORT}];
		\draw (-0.63*\LONG,1.84*\LP) .. controls (-0.6*\LONG,1.6*\LP) and (-0.55*\LONG,1.5*\LP) .. (-0.41*\LONG,1.33*\LP);
		\draw (0,\LP*1.6) [partial ellipse=246:360:{\LONG} and {\SHORT}];
		\draw (BL) -- (TL);
		\draw (BR) -- (TR);
		\draw (-\LONG,\LP) .. controls (-0.2*\LONG,0.7*\LP + 0.1*\RP) and (0.3*\LONG,0.1*\LP + 0.9*\RP) .. ({\LONG},\RP);
		\draw (-\LONG,\LP) .. controls (-0.7*\LONG,0.7*\LP + 0.4*\RP) and (0.1*\LONG,1.1*\RP) .. ({\LONG},\RP);
		\node at (0.22*\LONG,0.78*\HEIGHT/2) {$H_\tau$};
		\node[anchor=north west] at (0.3*\LONG,0.65*\HEIGHT) {$Q_\tau$};
		\node at (-0.3*\LONG,1.1*\HEIGHT/2) {$\Gamma_{S,\tau}$};
		\node[anchor=north west] at (\LONG,0.65*\HEIGHT) {$\Sigma_\tau$};
		\node at (0,0) {$\Omega$};
		\draw[->] (0,1.15*\HEIGHT/2) -- (-1,\HEIGHT-0.4) node[anchor=east]{$\nu$};
		\node[anchor=east] at (-\LONG,\LP) {$T_1$};
		\node[anchor=west] at (\LONG,\RP) {$T_2$};
	\end{tikzpicture}
	\caption{Left: the cylinder $Q$. \quad Right: the part of $Q$ underneath $\Gamma_S$.}
	\label{fig:geo-HR21}
\end{figure}
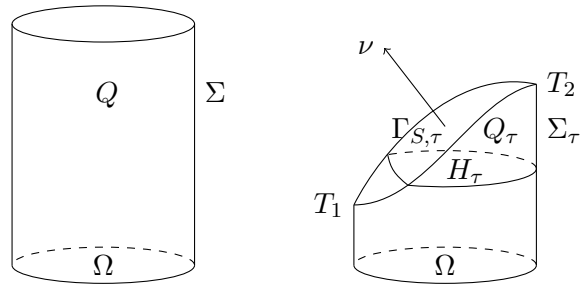

In what follows we use the notation $\langle T \rangle := (1+|T|^2)^{1/2}$ for simplicity.
Throughout this work we reserve the notation $T_2$ as
\[
T_2 := \sup \{ S(x) \,;\, x \in \Omega \}.
\]

\begin{Theorem} \label{thm:1H1-HR21}
	Given Assumption \ref{asp:nS-HR21}, then in the system \eqref{eq:1-HR21}-\eqref{eq:1con-HR21},
	the restriction of the solution $u$ on $\Gamma_S$ is in $H^1(\Gamma_S)$, and $\mathcal E(u; \Gamma_S)$ is well-defined.
	Moreover, there hold
	\begin{equation} \label{eq:1eH2-HR21}
		|\sqrt{\mathcal E(u; \Gamma_S)} - \sqrt{\mathcal E(u,0)}|
		\leq C \agl[T_2]^{1/2} (\nrm[H^1(\Sigma_0)]{f} + \nrm[L^2(Q_0)]{G}).
	\end{equation}
	for some constant $C$ depending only on $A$ and the dimension $n$, and is independent of $\Omega$.
\end{Theorem}

\begin{Remark} \label{rem:1S0-HR21}
	When $\Gamma_S$ is not horizontal ($S(x) \neq$ constant), we have $Q_0 \subsetneqq Q$ and $\Sigma_0 \subsetneqq \Sigma$, so the estimate given in Theorem \ref{thm:1H1-HR21} only requires parts of the data $f$ and $G$.
	When $\Gamma_S$ is strictly spacelike, i.e.~$|\nabla S(x)|_A \leq C < 1$ for all $x \in \overline \Omega$ for some constant $C$, Theorem \ref{thm:1H1-HR21} implies $u_t |_{\Gamma_S} \in L^2(\Gamma_S)$ as well.
\end{Remark}

%\begin{Remark} \label{rem:Tay-HR21}
%	In \cite[Chapter 2 Fig.~6.1, 6.2, 8.1]{taylor2011pa} Taylor studied spacelike cases where the lateral boundary are required to be empty sets,
%	but here we don't need this requirement.
%	Also, lightlike cases are not covered in \cite{taylor2011pa}.
%	And, energy estimates in \cite{taylor2011pa} are given in the form of \eqref{eq:1H1-HR21} below, while the inequality in Theorem \ref{thm:1H1-HR21} seems to be new.
%\end{Remark}

In many applications, it is more common to use energy rather than its square root, so we also present the following result which is a direct consequence of \eqref{eq:1eH2-HR21},
\begin{equation} \label{eq:1H1-HR21}
	\mathcal E(u; \Gamma_S)
	\leq C \mathcal E(u,0) + C \agl[T_2] (\nrm[H^1(\Sigma_0)]{f}^2 + \nrm[L^2(Q_0)]{G}^2).
\end{equation}
But readers should note that the estimate given in Theorem \ref{thm:1H1-HR21} is sharper than \eqref{eq:1H1-HR21}.

\smallskip

Energy can also be defined on different non-timelike hypersurfaces.
Denote $S_0(x) \equiv 0$ and $\Gamma_0 := \{ (x, S_0(x)) \,;\, x \in \overline \Omega \}$ so that $\Gamma_0 = \overline \Omega \times \{t = 0\}$.
And we define a family of hypersurfaces $\{\Gamma_\tau\}_{\tau > 0}$ (see Fig.~\ref{fig:Gtau-HR21}) by
\[
\Gamma_\tau := \{ (x, S_\tau(x)) \,;\, x \in \overline \Omega \},
\]
with $S_\tau$ satisfying the following requirement:
\begin{equation} \label{eq:Gat-HR21}
	S_\tau \in C^1(\overline \Omega) \text{~piecewisely}, \
	\text{such that for~} \tau_1 < \tau_2, \ S_{\tau_1}(x) \leq S_{\tau_2}(x) \ \forall x \in \overline \Omega.
\end{equation}

\begin{figure}[h]
	\centering
	\begin{tikzpicture}
	\draw (-1,0) -- (-1,4);
	\draw (1,0) -- (1,4);
	% 0
	\draw[dashed] (-1,0) -- (1,0);
	\node[anchor = south] (0,0) {$\Gamma_0$};
	% 1
	\draw[dashed] (-1,0.7) .. controls (-0.5, 1.0) and (0.5, 0.6) .. (1,1.3);
	\node at (0,1.2) {$\Gamma_{\tau_1}$};
	% 2
	\draw[dashed] (-1,1.9) .. controls (-0.5, 1.5) and (0.5, 2.2) .. (1,2.2);
	\node at (0,2.2) {$\Gamma_{\tau_2}$};
	% 3
	\draw[dashed] (-1,3) .. controls (-0.5, 3.2) and (0.3, 2.8) .. (0.45,2.7) .. controls (0.7,2.9) .. (1,3);
	\node at (0,3.2) {$\Gamma_{\tau_3}$};
	\node at (0,3.9) {$\vdots$};
	\end{tikzpicture}
	\caption{Non-timelike hypersurfaces}
	\label{fig:Gtau-HR21}
\end{figure}
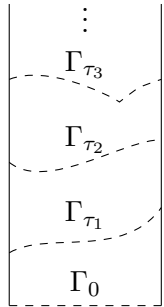

\noindent If for every fixed $x \in \overline \Omega$, the function $S_\tau(x)$ is continuous with respect to $\tau$,
then the family $\{\Gamma_\tau\}_{\tau \geq 0}$ will form another foliation of $\overline Q$ comparing to the standard foliation $\Omega \times \{t = \tau\}$ with $\tau \in [0,T]$.
We abbreviate the corresponding energy on $\Gamma_\tau$ as $\mathcal E(u; \Gamma_\tau)$, namely,
\[
\mathcal E(u; \Gamma_\tau) := \mathcal E(u; \Gamma_{S_\tau}).
\]
By Theorem \ref{thm:1H1-HR21}, the following corollary about $\mathcal E(u; \Gamma_\tau)$ is an immediate result.

\begin{Corollary} \label{cor:1en-HR21}
	Given Assumption \ref{asp:nS-HR21},
	and assuming the family $\{\Gamma_\tau\}|_{\tau \geq 0}$ satisfies \eqref{eq:Gat-HR21},
	then in the system \eqref{eq:1-HR21}-\eqref{eq:1con-HR21},
	when $G = 0$ and $f = 0$, the energy $\mathcal E(u; \Gamma_\tau)$ is well-defined and is conserved, i.e.~
	\[
	\mathcal E(u; \Gamma_\tau)
	= \frac 1 2 \int_\Omega (|\nabla u_0|_A^2 + |u_1|^2 + q |u_0|^2) \dif x, \ \ \text{for} \ \ \tau \geq 0.
	\]
\end{Corollary}

Corollary \ref{cor:1en-HR21} generalizes the classical energy conservation law, which says the energy is conserved on every horizontal surface $\Omega \times \{t = \tau\}$.
%It also generalizes the energy conservation in \cite[\S 2.5 \& \S 2.6]{taylor2011pa} to the complex-valued case.

\smallskip

We also obtain an estimate of the conormal derivative (with respect to $A$), 
\[
u_{\nu, A} := \nu_\Sigma \cdot A(x) \nabla u,
\]
where $\nu_\Sigma$ signifies the outer unit normal vector to $\Sigma$.

\begin{Theorem} \label{thm:2-HR21}
	Under the same assumptions as in Theorem \ref{thm:1H1-HR21}, we have
	\begin{equation*}
	\nrm[L^2(\Sigma_0)]{u_{\nu,A}}
	\leq C \langle T_2 \rangle^{1/2} \sqrt{\mathcal E(u,0)} + C \langle T_2 \rangle (\nrm[L^2(Q_0)]{G} + \nrm[H^1(\Sigma_0)]{f}),
	\end{equation*}
	for some constant $C$ depending only on $A$ and the dimension $n$.
\end{Theorem}

%\begin{Remark}
%	When $S(x) = \text{constant}$, the hypersurface $\Gamma_S$ will be horizontal and so $Q_0 = Q$ and $\Sigma_0 = \Sigma$.
%	In this case, the estimate about $u_{\nu,A}$ is proved in \cite[(4.7)]{llt1986non} (note that $|u_{\nu,A}| \leq C |u_\nu|$).
%	However, when $\Gamma_S$ is not horizontal, Theorem \ref{thm:2-HR21} seems to be new.
%\end{Remark}

Proofs of Theorems \ref{thm:1H1-HR21}, \ref{thm:2-HR21} and Corollary \ref{cor:1en-HR21} are presented in Section \ref{sec:app-HR21}.

\subsection{Motivation}

The classical result \cite[Remark 2.10]{llt1986non} says that the solution $u$ of \eqref{eq:1-HR21}-\eqref{eq:1con-HR21} satisfies
\[
u \in C([0,T]; H^1(\Omega)), \ \ \partial_t u \in C([0,T]; L^2(\Omega)),
\]
which implies
\begin{equation*}
	u \in H^1(Q).
\end{equation*}
Therefore, by the trace theorem, the restriction of $u$ to $\Gamma_S$ has regularity $H^{1/2}$.
However, when $\Gamma_S$ is horizontal, e.g.~when $\Gamma_S = \Omega \times \{t = \tau\}$ for some $\tau$, the restriction of $u$ on $\Gamma_S$ has $H^1$ regularity due to \cite[Remark 2.10]{llt1986non}.
From the point of view of the relativity of simultaneity in the theory of relativity \cite{wald84gen}, $\Omega \times \{t = \tau\}$ should not be more special than any other non-timelike hypersurface $\Gamma_S$.
Therefore, one would expect that the trace theorem for the solution $u$ is not sharp and $u$ shall also enjoy the same $H^1$ regularity on slanted $\Gamma_S$ as on $\Omega \times \{t = \tau\}$.
This is the motivation of this work.

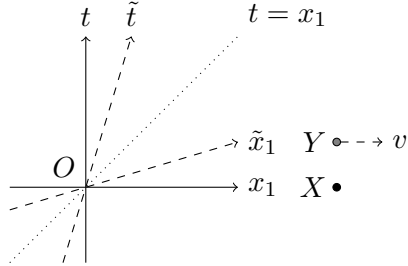
\begin{figure}[h]
	\centering
	\begin{tikzpicture}
		\node[anchor=south east] at (0,0) {$O$};
		\draw[dotted] (-1,-1) -- (2,2);
		\node[anchor=south west] at (2,2) {$t = x_1$};
		\draw[->] (-1,0) -- (2,0);
		\node[anchor=west] at (2,0) {$x_1$};
		\draw[->] (0,-1) -- (0,2);
		\node[anchor=south] at (0,2) {$t$};
		\draw[dashed, ->] (-1,-0.3) -- (2,0.6);
		\node[anchor=west] at (2,0.6) {$\tilde x_1$};
		\draw[dashed, ->] (-0.3,-1) -- (0.6,2);
		\node[anchor=south] at (0.6,2) {$\tilde t$};
		\filldraw[color=black!100, fill=black!100, thin](3.3,0) circle (0.05);
		\node[anchor=east] at (3.3,0) {$X$};
		\filldraw[color=black!100, fill=black!50, thin](3.3,0.6) circle (0.05);
		\draw[dashed, ->] (3.35,0.6) -- (3.9,0.6);
		\node[anchor=east] at (3.3,0.6) {$Y$};
		\node[anchor=west] at (3.9,0.6) {$v$};
	\end{tikzpicture}
	\caption{The observers $X$ and $Y$ stand at the origin $O$ at the time $\tau = 0$. $X$ stands still while $Y$ possesses an instantaneous velocity $v$ pointing to the right.
	$\{t, x_1\}$ is the instantaneous coordinate of $X$ and $\{\tilde t, \tilde x_1\}$ is that of $Y$.
	The time and space axes are perpendicular to each other for both $\{t, x_1\}$ and $\{\tilde t, \tilde x_1\}$, under the Minkowski metric $-\df t^2 + \df x_1^2$.}
\label{fig:rel-HR21}
\end{figure}

Let us explain the idea by a simplified example.
Assume there is a sound wave $u(x,t)$ propagating inside a domain $\Omega$ which satisfies the wave equation $(\partial_t^2 - \Delta) u = 0$, and the sound speed is normalized to $1$ in the medium, i.e.~$A(x) = 1$ in $\Omega$.
An observer $X$ is located inside $\Omega$ and he/she stands still relative to $\Omega$.
Another observer $Y$ is also located inside $\Omega$ but is moving in the direction of the first axis, say, $x_1$,
at a constant speed $v$ which is slower than that sound speed, $|v| < 1$.
See Fig.~\ref{fig:rel-HR21}.
We assume the $\Omega$ is large enough such that all the events mentioned here take place inside $\Omega \times [0,T]$ for $T$ large enough.
Then, from the perspective of $X$, the simultaneity at time $\tau$ is $\Omega \times \{t = \tau\}$,
while from the perspective of $Y$ the simultaneity is a slanted plane $\Gamma_S$ with $S(x) = vx$.
Let $(t, x)$ be the spacetime coordinate of $X$ and we denote $\gamma = (1 - v^2)^{-1/2}$ as the Lorentz factor.
According to the theory of relativity, the spacetime coordinate $(\tilde t, \tilde x)$ of $Y$ should satisfy
\[
\tilde t = \gamma (t - vx_1), \quad
\tilde x_1 = \gamma (x_1 - vt), \quad 
\tilde x_j = x_j~(j = 2, \cdots, n).
\]
The sound wave $\tilde{u}$ that $Y$ experienced should be $\tilde u(\tilde x, \tilde t) := u(x, t)$.
It can be checked that the wave equation is preserved under this Lorentz transformation, namely,
\[
(\partial_{\tilde t}^2 - \Delta_{\tilde x}) \tilde u(\tilde x, \tilde t)
= (\partial_{t}^2 - \Delta_{x}) u(x, t).
\]
This means that what $Y$ heard is also a wave which satisfies the same physical law with the sound that $X$ heard.
Theorem \ref{thm:1H1-HR21} tells us the profile of the sound (the wave shape across the space at a fixed time) that $Y$ heard has the same $H^1$ spacial regularity with the sound that $X$ heard.

\smallskip

This article is organized as follows.
In Section \ref{sec:pre-HR21} we make some preparations which are necessary for the subsequent analysis.
Section \ref{sec:est-HR21} is devoted to the proof of an intermediate result in which the solution is assumed to have $C^2$-smoothness.
Then the $C^2$-smoothness constraint is lifted in Section \ref{sec:app-HR21} by dealing with a compatibility issue and a regularity issue consecutively.
Finally, we prove the main results in Section \ref{sec:sqrt-HR21}.

\section{Some preparations} \label{sec:pre-HR21}

Throughout the article we denote by $C$ a generic constant whose value may varies from line to line.
We use $\divr_{x,t} (\vec a, b)$ to signify $\nabla \cdot \vec a + \partial_t b$.
The following lemma shall be used in Section \ref{sec:est-HR21} where the constraint $u \in C^2$ is stipulated.

\begin{Lemma} \label{lem:uu-HR21}
	Assume $u \in C^2$, $\vec \varphi(x,t) = (\varphi_1(x,t), \cdots, \varphi_n(x,t)) \in C^1(\R^{n+1}; \Rn)$ piecewisely and $q \in C^1(\Rn; \R)$, then at every $C^1$ continuous point of $\vec \varphi$, we have the following identities:
	\begin{align}
	2 \Re \{\overline{u}_t [\partial_t^2 u - & \nabla \cdot (A(x) \nabla u) + qu] \} \nonumber \\
	= & \, \Re \divr_{x,t} \big[ -2 \overline{u}_t A(x) \nabla u, |u_t|^2 + |\nabla u|_A^2 + q(x)|u|^2 \big], \label{eq:uu-HR21} \\
	2 \Re \{ (\vec \varphi \cdot \nabla \overline u) & [\partial_t^2 u - \nabla \cdot (A(x) \nabla u) + qu] \} \nonumber \\
	= & \, \Re \divr_{x,t} \big[ \vec \varphi (|\nabla u|_A^2 - |u_t|^2 + q|u|^2) - 2(\vec \varphi \cdot \nabla \overline u) A\nabla u, 2 (\vec \varphi \cdot \nabla \overline u) u_t \big] \nonumber \\
	& - ( \nabla \cdot \vec \varphi) (|\nabla u|_A^2 - |u_t|^2 + q|u|^2) - 2 \vec \varphi_t \cdot \Re ( \nabla \overline u u_t) - (\vec \varphi \cdot \nabla A)(\nabla u, \nabla u) \nonumber \\
	& + 2 (\partial_j \varphi_k) \Re (\overline u_k a_{jl} u_l) - (\nabla q \cdot \vec \varphi) |u|^2, \label{eq:uu3-HR21}
	\end{align}
	where $\Re$ stands for the real part, and the summation convention is called for the last term.
\end{Lemma}

\begin{proof}
	A straightforward computation shows
	\begin{equation*}
	2 \Re \{\overline{u}_t \partial_t^2 u \}
	= (|u_t|^2)', \quad
	2 \Re \{\overline{u}_t \nabla \cdot (A \nabla u) \}
	= 2 \Re \nabla \cdot (\overline{u}_t A \nabla u) - (|\nabla u|_A^2)',
	\end{equation*}
	\[
	2 \Re \{\overline{u}_t q u \}
	= 2 q \Re \{\overline{u}_t u \}
	= q \partial_t(|u|^2)
	\]
	which give \eqref{eq:uu-HR21}.

	Similarly, we can show
	\begin{align}
	& \, 2 \Re \{ (\vec \varphi \cdot \nabla \overline u) \partial_t^2 u \}
	= \Re \partial_t [ 2(\vec \varphi \cdot \nabla \overline u) u_t ] - \vec \varphi \cdot 2 \Re ( u_t \nabla \overline u_t) - 2 \vec \varphi_t \cdot \Re ( \nabla \overline u u_t) \nonumber \\
	= & \, \Re \partial_t [ 2(\vec \varphi \cdot \nabla \overline u) u_t ] - \vec \varphi \cdot \nabla (|u_t|^2) - 2 \vec \varphi_t \cdot \Re ( \nabla \overline u u_t) \nonumber \\
	= & \, \Re \partial_t [ 2(\vec \varphi \cdot \nabla \overline u) u_t ] - \nabla \cdot (\vec \varphi |u_t|^2) + (\nabla \cdot \vec \varphi) |u_t|^2 - 2 \vec \varphi_t \cdot \Re ( \nabla \overline u u_t). \label{eq:vpt-HR21}
	\end{align}
	$\vec \varphi$ is a real-valued vector function, so we also have
	\begin{align}
	& 2 \Re \{ (\vec \varphi \cdot \nabla \overline u) \nabla \cdot (A \nabla u) \}
	= \Re \nabla \cdot [ 2(\vec \varphi \cdot \nabla \overline u) A \nabla u ] - 2\Re [\nabla (\vec \varphi \cdot  \nabla \overline u) \cdot A \nabla u] \nonumber \\
	= & \Re \nabla \cdot [ 2(\vec \varphi \cdot \nabla \overline u) A \nabla u ] - 2\Re [(\partial_j \varphi_k) \overline u_k a_{jl} u_l] - 2\Re[\varphi_k \overline u_{jk} a_{jl} u_l]. \label{eq:npn1-HR21}
	\end{align}
	We compute ``$\varphi_k \overline u_{jk} a_{jl} u_l$'' as follows,
	\begin{align}
		\varphi_k \overline u_{jk} a_{jl} u_l
		& = \varphi_k \partial_k ( \overline u_j a_{jl} u_l ) - \varphi_k \overline u_j (\partial_k a_{jl}) u_l - \varphi_k \overline u_j a_{jl} u_{lk} \nonumber \\
		& = \vec \varphi \cdot \nabla ( |\nabla u|_A^2 ) - (\vec \varphi \cdot \nabla A)(\nabla u, \nabla u) - \varphi_k u_{jk} a_{lj} \overline u_l \nonumber \\
		& = \nabla \cdot (\vec \varphi  |\nabla u|_A^2 ) - (\nabla \cdot \vec \varphi) |\nabla u|_A^2 - (\vec \varphi \cdot \nabla A)(\nabla u, \nabla u) - \varphi_k u_{jk} a_{lj} \overline u_l. \label{eq:pjk-HR21}
	\end{align}
	Because $A$ is symmetric and $A$, $\vec \varphi$ are real-valued, we can conclude from \eqref{eq:pjk-HR21} that
	\begin{equation} \label{eq:Rpk-HR21}
		2\Re[ \varphi_k \overline u_{jk} a_{jl} u_l ]
		= \nabla \cdot (\vec \varphi  |\nabla u|_A^2 ) - (\nabla \cdot \vec \varphi) |\nabla u|_A^2 - (\vec \varphi \cdot \nabla A)(\nabla u, \nabla u).
	\end{equation}
	Combining \eqref{eq:npn1-HR21} and \eqref{eq:Rpk-HR21}, we arrive at
	\begin{align}
	2 \Re \{ (\vec \varphi \cdot \nabla \overline u) \nabla \cdot (A \nabla u) \}
	& = \Re \nabla \cdot [ 2(\vec \varphi \cdot \nabla \overline u) A \nabla u ] - 2 (\partial_j \varphi_k) \Re (\overline u_k a_{jl} u_l) \nonumber \\
	& \quad - \nabla \cdot (\vec \varphi  |\nabla u|_A^2 ) + (\nabla \cdot \vec \varphi) |\nabla u|_A^2 + (\vec \varphi \cdot \nabla A)(\nabla u, \nabla u). \label{eq:vpx-HR21}
	\end{align}
	Also,
	\begin{align}
		2 \Re \{ (\vec \varphi \cdot \nabla \overline u) qu \}
		& = (q \vec \varphi \cdot \nabla \overline u) u + (q \vec \varphi \cdot \nabla u) \overline u
		= q \vec \varphi \cdot \nabla (\overline u u) \nonumber \\
		& = \nabla \cdot (q \vec \varphi |u|^2) - [q(\nabla \cdot \vec \varphi) + \nabla q \cdot \vec \varphi] |u|^2. \label{eq:vpq-HR21}
	\end{align}
	Subtracting \eqref{eq:vpx-HR21} from \eqref{eq:vpt-HR21} and adding \eqref{eq:vpq-HR21}, we arrive at \eqref{eq:uu3-HR21}.
	The proof is done.
\end{proof}

In the following lemma we abbreviate $\mathcal E(u; \Gamma_S)$ as $\mathcal E(u)$ for simplicity,
and we show the energy possesses a similar triangle inequality property.

\begin{Lemma} \label{lem:tri-HR21}
	For any $u_1$, $u_2$ such that $\mathcal E(u_1)$, $\mathcal E(u_2)$ are well-defined, we have $\mathcal E(u_1 - u_2)$ is also well-defined and
	\begin{align}
	\sqrt{\mathcal E(u_1 + u_2)}
	& \leq \sqrt{\mathcal E(u_1)} + \sqrt{\mathcal E(u_2)}, \label{eq:tri1-HR21} \\
	|\mathcal E(u_1) - \mathcal E(u_2)|
	& \leq \mathcal E(u_1 - u_2) + 2 \sqrt{\mathcal E(u_1 - u_2)} \sqrt{\mathcal E( u_j)}, \quad j = 1,2. \label{eq:tri2-HR21}
	\end{align}
\end{Lemma}

\begin{proof}
	We show \eqref{eq:tri1-HR21} by direct computation.
	From \eqref{eq:en-HR21} we have
	\begin{align*}
	& \, \mathcal E(u_1 + u_2) \\
	= & \, \frac 1 2 \int_{\Omega} \big[ |\nabla \big( u_1 + u_2 \big)|_A^2 + (1 - |\nabla S(x)|_A^2) |\partial_t(u_1 + u_2)|^2 + q(x)|u_1 + u_2|^2 \big] \dif x \\
	= & \, \mathcal E(u_1) + \mathcal E(u_2) + 2\Re \int_{\Omega} \big[ \nabla \big( \overline{u}_1 \big) \cdot A \nabla \big( u_2 \big) + (1 - |\nabla S|_A^2) \overline{\partial_t u}_1 \partial_t u_2 + q(x) \overline{u}_1 u_2 \big] \dif x \\
	\leq & \, \mathcal E(u_1) + \mathcal E(u_2) + 2 \int_{\Omega} \big[ |\nabla \big( u_1 \big)|_A \cdot |\nabla \big( u_2 \big)|_A + (1 - |\nabla S|_A^2) |\partial_t u_1| |\partial_t u_2| + q |u_1 u_2| \big] \dif x \\
	\leq & \, \mathcal E(u_1) + \mathcal E(u_2) + 2 \int_{\Omega} \big[ |\nabla \big( u_1 \big)|_A^2 + (1 - |\nabla S|_A^2) |\partial_t u_1|^2 + q |u_1|^2 \big]^{1/2} \\
	& \hspace{8.5em} \cdot \big[ |\nabla \big( u_2 \big)|_A^2 + (1 - |\nabla S|_A^2) |\partial_t u_2|^2 + q |u_2|^2 \big]^{1/2} \dif x \\
	\leq & \, \mathcal E(u_1) + \mathcal E(u_2) + 2 \sqrt{\mathcal E(u_1)} \sqrt{\mathcal E(u_2)}
	= \big( \sqrt{\mathcal E(u_1)} + \sqrt{\mathcal E(u_2)} \big)^2.
	\end{align*}
	Note that we implicitly used conditions $|\nabla S(x)|_A \leq 1$ and $q(x) \geq 0$.
	\eqref{eq:tri1-HR21} is proved.

	Next we show \eqref{eq:tri2-HR21}.
	Similar to $|a + b| \leq |a| + |b|$, we also have $|a + b|_A \leq |a|_A + |b|_A$.
	This is because
	\begin{align*}
	|a + b|_A^2
	& = |a|_A^2 + |b|_A^2 + 2\Re \{\overline a \cdot Ab\}
	= |a|_A^2 + |b|_A^2 + 2\Re \{\overline {A^{1/2}a} \cdot A^{1/2}b\} \\
	& \leq |a|_A^2 + |b|_A^2 + 2|a|_A |b|_A
	= (|a|_A + |b|_A)^2.
	\end{align*}
	By this triangle inequality, it is straightforward to check $\mathcal E(u_1 - u_2) \leq 2[\mathcal E(u_1) + \mathcal E(u_2)]$, so $\mathcal E(u_1 - u_2)$ is well-defined.
	
	The triangle inequality obtained above also gives $|a|_A - |b|_A \leq |a-b|_A$.
	Hence
	\begin{align*}
	|\nabla (u_1)|_A^2 - |\nabla (u_2)|_A^2
	& = \big[ |\nabla (u_1)|_A - |\nabla (u_2)|_A \big]^2  + 2\big[ |\nabla (u_1)|_A - |\nabla (u_2)|_A \big]|\nabla (u_2)|_A \\
	& \leq |\nabla (u_1 - u_2)|_A^2 + 2|\nabla (u_1 - u_2)|_A |\nabla (u_2)|_A.
	\end{align*}
	Similarly,
	\begin{align*}
	|\partial_t u_1|^2 - |\partial_t u_2|^2
	& \leq |\partial_t (u_1 - u_2)|^2 + 2|\partial_t (u_1 - u_2)| |\partial_t u_2| \\
	q|u_1|^2 - q|u_2|^2
	& \leq q|u_1 - u_2|^2 + 2q|u_1 - u_2| |u_2|
	\end{align*}
	Summing up these three inequalities gives
	\begin{align*}
	& \, \mathcal E(u_1) - \mathcal E(u_2) \\
	\leq & \, \mathcal E(u_1 - u_2) + \int \Big[ |\nabla (u_1 - u_2)|_A |\nabla (u_2)|_A + (1 - |\nabla S(x)|_A^2) |\partial_t (u_1 - u_2)| |\partial_t u_2| \\
	& \quad \quad \quad \quad \quad \quad \quad + q|u_1 - u_2| |u_2| \Big] \dif x \\
	\leq & \, \mathcal E(u_1 - u_2) + \int \sqrt{|\nabla (u_1 - u_2)|_A^2 + (1 - |\nabla S(x)|_A^2)|\partial_t (u_1 - u_2)|^2 + q|u_1 - u_2|^2} \\
	& \quad \quad \quad \quad \quad \quad \quad \times \sqrt{|\nabla (u_2)|_A^2 + (1 - |\nabla S(x)|_A^2)|\partial_t u_2|^2 + q|u_2|^2} \dif x \\
	\leq & \, \mathcal E(u_1 - u_2) + 2 \sqrt{\mathcal E(u_1 - u_2)} \sqrt{\mathcal E( u_2)}.
	\end{align*}
	Similar arguments also imply
	\[
	\mathcal E(u_1) - \mathcal E(u_2)
	\leq \mathcal E(u_1 - u_2) + 2 \sqrt{\mathcal E(u_1 - u_2)} \sqrt{\mathcal E( u_1)}.
	\]
	Therefore,
	\[
	|\mathcal E(u_1) - \mathcal E(u_2)|
	\leq \mathcal E(u_1 - u_2) + 2 \sqrt{\mathcal E(u_1 - u_2)} \sqrt{\mathcal E( u_j)}, \quad j = 1,2.
	\]
	The proof is done.
\end{proof}

\subsection{The decomposition of the gradient} \label{subsec:nd-HR21}

This part is devoted to the analysis of the relation between $|\nu_\Sigma|_A^2 |\nabla u|_A^2$ and $|u_{\nu,A}|^2$ which will appear in \eqref{eq:igI11-HR21}.
Here $\nu_\Sigma$ stands for the outer unit normal vector to $\Sigma$,
and $u_{\nu,A}$ signifies the conormal derivative with respect to $A$, i.e.~$u_{\nu,A} := \nu_\Sigma \cdot A \nabla u$.
Readers can skip this part for the first time.

Recall that $\Sigma = \partial \Omega \times (0,T)$.
When $\Omega$ is the unit ball and $A(x)$ is the identity matrix, by straightforward computations it can be checked that on $\Sigma$ we have the following identity:
\begin{equation} \label{eq:nuX-HR21}
	|\nu_\Sigma|_A^2 |\nabla u|_A^2 = |u_{\nu,A}|^2 + \frac 1 2 \sum_{i \neq j} |X_{ij} u|^2, \quad \text{where} \quad X_{ij} = x_i \partial_j - x_j \partial_i, \ |x| = 1,
\end{equation}
see e.g.~\cite[(1.19)]{RakeshSalo1}.
Note that here $X_{ij}$ are vector fields tangential to the sphere of the unit ball.
Therefore, the square norm of the gradient $|\nabla u|^2$ is decomposed into the desired term $|u_{\nu,A}|^2$ along with other terms which are tangential gradients on $\Sigma$.
For general domain $\Omega$ and general matrix $A(x)$, we can show a similar decomposition result.

\begin{Lemma} \label{lem:nd-HR21}
	Assume $\Omega$ is a $C^1$-smooth domain.
	Then for a $C^1$-smooth function $u$, there exist two constants $C$ depending only on $c_1$, $c_2$ and the dimension $n$ such that on $\Sigma$ we have
	\begin{equation} \label{eq:nd-HR21}
		|\nu_\Sigma|_A^2 |\nabla u|_A^2
		\leq 1 \times |u_{\nu,A}|^2 + C |u_{\nu,A}| |\nabla_\Sigma u| +  C |\nabla_\Sigma u|^2,
	\end{equation}
	where $\nabla_\Sigma u$ represents the tangential gradient of $u$ on $\Sigma$.
\end{Lemma}

\begin{Remark}
	In Lemma \ref{lem:nd-HR21} we emphasize the coefficient of the term $|u_{\nu,A}|^2$ is exactly $1$.
\end{Remark}

\begin{proof}[Proof of Lemma \ref{lem:nd-HR21}]
	
	Denote $e_1 = \nu_\Sigma$ and fix an orthonormal basis $\{e_2, \cdots, e_{n+1}\}$ of a local chart of $\Sigma$.
	Then $\{e_1, e_2, \cdots, e_{n+1}\}$ is also an locally orthonormal basis of $\R^{n+1}$.
	The conclusion \eqref{eq:nd-HR21} is a local estimate so local arguments are enough for the proof.
	
	For simplicity we denote
	\begin{equation*}
		E(x) = (A e_1,\, e_2 ,\, \cdots ,\, e_{n+1}).
	\end{equation*}
	The ellipticity condition \eqref{eq:ellp-HR21} guarantees that $E(x)$ is always invertible;
	this is because $e_1 \cdot A e_1 \geq c_1 |e_1|^2 = c_1 > 0$, so $A e_1$ always has nonzero component in $e_1$ direction, and so $\{A e_1,\, e_2 ,\, \cdots ,\, e_{n+1}\}$ is always linearly independent.
	Hence, thanks to the existence of $E^{-1}(x)$, we can compute
	\begin{align*}
		& \, |\nabla u|_A^2
		= (\nabla \overline u)^T A \nabla u
		= (\nabla \overline u)^T E(x) [A^{-1/2} E(x)]^{-1} [E^T(x) A^{-1/2}]^{-1} E^T(x) \nabla u \\
		= & \, (\nabla \overline u)^T E(x) [E^{-1}(x) A E^{-1,T}(x)] E^T(x) \nabla u \\
		= & \, (\overline u_{\nu,A},\, e_2 \cdot \nabla \overline u,\, \cdots,\, e_{n+1} \cdot \nabla \overline u) [E^{-1}(x) A E^{-1, T}(x)] (u_{\nu,A},\, e_2 \cdot \nabla u,\, \cdots,\, e_{n+1} \cdot \nabla u)^T.
	\end{align*}
	Here $E^{-1,T}(x)$ signifies the transpose of the inverse of the matrix $E(x)$.
	Let us denote
	\[
	M(x) = E^{-1}(x) A(x) E^{-1,T}(x)
	\]
	and use $M_{ij}$ to signify the elements of $M$.
	Note that $M$ is symmetric, then we have
	\begin{equation} \label{eq:nuAM-HR21}
		|\nabla u|_A^2
		= M_{11} |u_{\nu,A}|^2 + 2\Re \{\overline u_{\nu,A} \sum_{j = 2}^n M_{1j} (e_j \cdot \nabla u)\} + \sum_{k,l = 2}^n M_{kl} (e_k \cdot \nabla \overline u) (e_l \cdot \nabla u).
	\end{equation}

	We claim that $e_l \cdot \nabla u$ is a component of the tangential gradient of $u$ on $\Sigma$, which is similar to the $X_{ij} u$ term in \eqref{eq:nuX-HR21}.
	To see that, we choose a $C^1$-smooth curve $\gamma \colon (-1,1) \to \Sigma$ on $\Sigma$ satisfying $\gamma(0) = x$ and $\dot \gamma(0) = e_l$,
	then $e_l \cdot \nabla u(x)$ can be represented as $\frac{\df}{\df t} |_{t = 0} \big( u(\gamma(t)) \big)$.
	This justifies our claim.
	Hence, we have
	\[
	|e_l \cdot \nabla u| \leq |\nabla_\Sigma u|, \quad l = 2, \cdots, n+1,
	\]
	and so we can continue \eqref{eq:nuAM-HR21} as
	\begin{align}
		|\nu_\Sigma|_A^2 |\nabla u|_A^2
		& = |\nu_\Sigma|_A^2 (M_{11} |u_{\nu,A}|^2 + C |u_{\nu,A}| |\nabla_\Sigma u| + C |\nabla_\Sigma u|^2) \nonumber \\
		& \leq (|\nu_\Sigma|_A^2 M_{11}) |u_{\nu,A}|^2 + C |u_{\nu,A}| |\nabla_\Sigma u| + C |\nabla_\Sigma u|^2, \label{eq:nMC-HR21}
	\end{align}
	for some constant $C$.
	From the definition of the matrix $M$, it can be checked that $C$ depends only on $c_1$, $c_2$ and $n$.
	
	It is left to show
	\begin{equation} \label{eq:M1n-HR21}
		|\nu_\Sigma|_A^2 M_{11} = 1.
	\end{equation}
	To see this, let us represent the inverse matrix $E^{-1}(x)$ as
	\[
	E^{-1}(x) = (f_1(x), f_2(x), \cdots, f_{n+1}(x))^T,
	\]
	then the matrix identity $E^{-1}(x) E(x) = I$ gives
	\begin{align}
		& f_1(x) \cdot e_j(x) = 0, \ j = 2, \cdots, n+1. \label{eq:fe2-HR21} \\
		& f_1(x) \cdot A(x) e_1(x) = 1. \label{eq:fe1-HR21}
	\end{align}
	Because $\{ e_1, e_2, \cdots, e_{n+1} \}$ is an orthonormal basis in a chart, from \eqref{eq:fe2-HR21} we see $f_1(x)$ is parallel to $e_1(x)$, i.e.~$f_1(x) = \lambda(x) e_1(x)$ for some function $\lambda(x)$.
	Substitute this into \eqref{eq:fe1-HR21}, we see $\lambda(x) = |e_1(x)|_A^{-2} = |\nu_\Sigma|_A^{-2}$.
	Therefore, we have
	\begin{equation*}
		M_{11} |\nu_\Sigma|_A^2
		= f_1^T(x) A(x) f_1(x) |\nu_\Sigma|_A^2
		= \lambda(x) e_1^T(x) A(x) \lambda(x) e_1(x) |\nu_\Sigma|_A^2
		= \lambda^2(x) |\nu_\Sigma|_A^4
		= 1,
	\end{equation*}
	which is \eqref{eq:M1n-HR21}.
	Combining \eqref{eq:nMC-HR21} and \eqref{eq:M1n-HR21}, we can complete the proof.
\end{proof}

\section{The \texorpdfstring{$C^2$}{C2} smooth case} \label{sec:est-HR21}

Recall $\agl[T] := (1+|T|^2)^{1/2}$ and $\mathcal E(u,0) := \mathcal E(u; S(x) = 0)$.
In this section we aim to prove the following result.

\begin{Proposition} \label{prop:Bddd-HR21}
	Given Assumption \ref{asp:nS-HR21}.
	Assume $u \in C^2(\overline Q)$ solves \eqref{eq:1-HR21}-\eqref{eq:1con-HR21}.
	Then we have
	\begin{align*}
	|\mathcal E(u; \Gamma_S) - \mathcal E(u,0)|
	& \leq C \agl[T_2]^{1/2} \big[ \sqrt{\mathcal E(u,0)} + \agl[T_2]^{1/2} (\nrm[H^1(\Sigma_0)]{f} + \nrm[L^2(Q_0)]{G}) \big] \nonumber \\
	& \quad \times (\nrm[H^1(\Sigma_0)]{f} + \nrm[L^2(Q_0)]{G}),
	\end{align*}
	for some constant $C$ depending only on $A$ and the dimension $n$.
	Especially, when $f = 0$ and $G = 0$ in \eqref{eq:1con-HR21}, we have
	\begin{equation} \label{eq:Bdd2-HR21}
	\mathcal E(u; \Gamma_S)
	= \frac 1 2(\nrm[L^2(\Omega)]{\nabla u_0}^2 + \nrm[L^2(\Omega)]{u_1}^2 + \nrm[L^2(\Omega)]{\sqrt q u_0}^2).
	\end{equation}
\end{Proposition}

This section and a major portion of Section \ref{sec:app-HR21} involve the same hypersurface $\Gamma_S$, so we abbreviate $\mathcal E(u; \Gamma_S)$ as $\mathcal E(u)$ for short in these parts.
Recall the partial hypersurface $\Gamma_{S,\tau}$ given in Fig.~\ref{fig:geo-HR21}.
For technical reasons we also introduce the following functional $\mathcal E_\tau(u)$ which we shall call it \emph{partial energy}
and which takes a real number $\tau$ as its parameter,
{\small \begin{equation} \label{eq:ent-HR21}
	\mathcal E_\tau(u)
	:= \frac 1 2 \int_{\pi(\Gamma_{S, \tau})} \big[ |\nabla \big( u(x, S(x)) \big)|_A^2 + (1 - |\nabla S(x)|_A^2) |u_t(x, S(x))|^2 + q(x) |u(x,\tau)|^2 \big] \dif x,
\end{equation}}
where $\pi \colon (x,t) \mapsto x$ is a projection map.
Readers can compare \eqref{eq:ent-HR21} with \eqref{eq:en-HR21}, and shall distinguish the notation $\mathcal E_\tau(u)$ with $\mathcal E(u; \Gamma_\tau)$ defined on different hypersurfaces $\Gamma_\tau$.
When $\tau \leq T_1$, we have $\pi(\Gamma_{S, \tau}) = \Omega$, so $\mathcal E_\tau(u)$ will be a constant with respect to $\tau$ when $\tau \leq T_1$.
Under Assumption \ref{asp:nS-HR21} and $q(x) \geq 0$,
we always have $\mathcal E_\tau(u) \leq \mathcal E(u)$;
and when $\tau \leq T_1$ we have $\mathcal E_\tau(u) = \mathcal E(u)$.

The relationship between $\nrm[H^1(\Gamma_S)]{u}$ and $\mathcal E(u)$ are given below,
\begin{align}
	\nrm[H^1(\Gamma_S)]{u}^2
	& = \int_{\Omega} |\nabla \big( u(x, S(x)) \big)|^2 \frac {\dif x} {(1 + |\nabla S(x)|^2)^{1/2}}
	\leq c_1^{-1} \int_{\Omega} |\nabla \big( u(x, S(x)) \big)|_A^2 \dif x \nonumber \\
	& = c_1^{-1} \int_{\Omega} (|\nabla u|_A^2 + |\nabla S|_A^2 |u_t|^2 + 2 \Re \{\nabla S \cdot \overline u_t A \nabla u\}) \dif x \nonumber \\
	& \leq c_1^{-1} \int_{\Omega} (|\nabla u|_A^2 + c_2 |u_t|^2 + 2 \Re \{\nabla S \cdot \overline u_t A \nabla u\}) \dif x \nonumber \\
	& \leq 2c_1^{-1} \max\{1, c_2\} \mathcal E(u), \label{eq:H1H-HR21}
\end{align}
where we used Assumption \ref{asp:nS-HR21} and \eqref{eq:ellp-HR21}
For readers convenience we also record the following identity,
\[
|\nabla \big( u \big)|_A^2 + (1 - |\nabla S|_A^2) |u_t|^2 + q|u|^2
= |\nabla u|_A^2 + |u_t|^2 + q|u|^2 + 2 \Re \{\overline u_t \nabla S \cdot A \nabla u\}.
\]

The arguments of proving Proposition \ref{prop:Bddd-HR21} are divided into several steps.

\begin{Lemma} \label{lem:Bdd-HR21}
	Under the same condition as in Proposition \ref{prop:Bddd-HR21}, we have
	\begin{align*}
	\mathcal E_\tau(u)
	& \leq \int_{H_\tau} (|\nabla u(x,\tau)|_A^2 + |u_t(x,\tau)|^2 + q(x)|u(x,\tau)|^2) \dif x \nonumber \\
	& \quad + 3 \nrm[L^2(\Sigma_\tau)]{f_t} \nrm[L^2(\Sigma_\tau)]{u_{\nu,A}} + 4 (T_2 - \tau) \nrm[L^2(Q_\tau)]{G}^2.
	\end{align*}
\end{Lemma}

\begin{proof}
	Equation \eqref{eq:1-HR21} gives $\partial_t^2 u - \nabla \cdot (A(x) \nabla u) + q u = G$ in $Q_\tau$.
	Hence, integrating the identity \eqref{eq:uu-HR21} in $Q_\tau$, we can have
	\begin{align}
	& \Re \int_{Q_\tau} 2 \overline{u}_t G
	= \Re \int_{\partial Q_\tau} \nu_{Q_{\tau}} \cdot (-2 \overline{u}_t A \nabla u,|u_t|^2 + |\nabla u|_A^2 + q|u|^2) \nonumber \\
	= & \Re \int_{H_\tau} (\vec 0, -1) \cdot (-2 \overline{u}_t A \nabla u,|u_t|^2 + |\nabla u|_A^2 + q|u|^2) \nonumber \\
	& + \Re \int_{\Sigma_\tau} \nu_\Sigma \cdot (-2 \overline{u}_t A \nabla u,|u_t|^2 + |\nabla u|_A^2 + q|u|^2) \nonumber \\
	& + \Re \int_{\Gamma_{S, \tau}} \nu \cdot (-2 \overline{u}_t A \nabla u, |u_t|^2 + |\nabla u|_A^2 + q|u|^2) \qquad (\nu \text{~is defined in~} \eqref{eq:nu-HR21}) \nonumber \\
	= & -\int_{H_\tau} (|u_t|^2 + |\nabla u|_A^2 + q|u|^2) \dif \sigma - 2 \Re \int_{\Sigma_\tau}  \overline{u}_t u_{\nu,A} \dif \sigma \qquad \big( u_{\nu,A} := \nu_\Sigma \cdot A \nabla u \big) \nonumber \\
	& + \Re \int_{\Gamma_{S, \tau}} (2 \nabla S \cdot \overline u_t A \nabla u + |u_t|^2 + |\nabla u|_A^2 + q|u|^2) \frac {\dif \sigma} {\sqrt{1 + |\nabla S|^2}} \nonumber \\
	= & -\int_{H_\tau} (|u_t|^2 + |\nabla u|_A^2 + q|u|^2) \dif \sigma - 2 \Re \int_{\Sigma_\tau}  \overline{u}_t u_{\nu,A} \dif \sigma \nonumber \\
	& + \int_{\pi(\Gamma_{S, \tau})} (2 \Re \{\nabla S \cdot \overline u_t A \nabla u\} + |u_t|^2 + |\nabla u|_A^2 + q|u|^2) \dif x \nonumber \\
	= & -\int_{H_\tau} (|u_t|^2 + |\nabla u|_A^2 + q|u|^2) \dif \sigma - 2 \Re \int_{\Sigma_\tau} \overline{u}_t u_{\nu,A} \dif \sigma + \mathcal E_\tau(u), \label{eq:0g0-HR21}
	\end{align}
	where $\nu_{Q_{\tau}}$ signifies the outer unit normal of $Q_{\tau}$,
	$\nu_\Sigma$ is the outer unit normal vector to $\Sigma$, and
	$u_{\nu,A}$ signifies the conormal derivative with respect to $A$, i.e.~$u_{\nu,A} := \nu_\Sigma \cdot A \nabla u$.

	We denote
	\begin{equation} \label{eq:0gh-HR21}
		e(\tau) := \frac 1 2 \int_{H_\tau} (|\nabla u|_A^2 + |u_t|^2 + q|u|^2) \dif x.
	\end{equation}
	Substituting \eqref{eq:0gh-HR21} into \eqref{eq:0g0-HR21}, we have
	\begin{align}
		e(\tau)
		& = -\Re \int_{Q_\tau} \overline{u}_t G + \mathcal E_\tau(u) - \Re \int_{\Sigma_\tau} \overline{u}_t u_{\nu,A} \dif \sigma \nonumber \\
		& \leq  \int_{Q_\tau} \frac 1 {2K} |u_t|^2 + \int_{Q_\tau} \frac K 2 |G|^2 + \mathcal E_\tau(u) - \Re \int_{\Sigma_\tau} \overline{u}_t u_{\nu,A} \dif \sigma \nonumber \\
		& \leq \frac 1 K \int_\tau^{T_2} e(s) \dif s + \frac K 2 \nrm[L^2(Q_\tau)]{G}^2 + \mathcal E_\tau(u) + \int_{\Sigma_\tau}  |u_t u_{\nu,A}| \dif \sigma, \nonumber
	\end{align}
	where $K$ can be any positive number.
	Recall that under Assumption \ref{asp:nS-HR21} we have $|\nabla S|_A \leq 1$, so $\mathcal E_\tau(u)$ is always non-increasing with respect to $\tau$.
	Hence, by Gr\"onwall's inequality we can obtain
	\begin{equation} \label{eq:0g12-HR21}
	e(\tau)
	\leq e^{(T_2 - \tau)/K} (\frac K 2 \nrm[L^2(Q_\tau)]{G}^2 + \mathcal E_\tau(u) + \int_{\Sigma_\tau}  |u_t u_{\nu,A}| \dif \sigma), \quad \forall \tau \in [0, T_2].
	\end{equation}
	
	On the other hand, when $\tau \leq T_1$, from \eqref{eq:0g0-HR21} we also obtain
	\begin{align*}
	& \, \mathcal E_\tau(u)
	= e(\tau) + \Re \int_{\Sigma_\tau}  \overline{u}_t u_{\nu,A} \dif \sigma + \Re \int_{Q_\tau} \overline{u}_t G \nonumber \\
	\leq & \, e(\tau) + \int_{\Sigma_\tau}  |u_t u_{\nu,A}| \dif \sigma + \frac 1 {2\epsilon} \int_{Q_\tau} |G|^2 + \epsilon \int_\tau^{T_2} e(s) \dif s \nonumber \\
	\leq & \, e(\tau) + \int_{\Sigma_\tau}  |u_t u_{\nu,A}| \dif \sigma + \frac 1 {2\epsilon} \nrm[L^2(Q_\tau)]{G}^2 \nonumber \\
	& \, + \epsilon \int_\tau^{T_2} e^{(T_2 - s)/K} \big[ \frac K 2 \nrm[L^2(Q_s)]{G}^2 + \mathcal E_s(u) + \int_{\Sigma_s} |u_t u_{\nu,A}| \dif \sigma \big] \dif s \qquad (\text{by~} \eqref{eq:0g12-HR21}) \nonumber \\
	\leq & \, e(\tau) + \int_{\Sigma_\tau}  |u_t u_{\nu,A}| \dif \sigma + \frac 1 {2\epsilon} \nrm[L^2(Q_\tau)]{G}^2 \nonumber \\
	& \, + \epsilon K (e^{(T_2 - \tau)/K} - 1) \big[ \frac K 2 \nrm[L^2(Q_\tau)]{G}^2 + \mathcal E_\tau(u) + \int_{\Sigma_\tau} |u_t u_{\nu,A}| \dif \sigma \big] \nonumber \\
	\leq & \, e(\tau) + [1 + \epsilon K (e^{(T_2 - \tau)/K} - 1)] \int_{\Sigma_\tau}  |u_t u_{\nu,A}| \dif \sigma \nonumber \\
	& \, + \frac 1 2 [\frac 1 \epsilon + \epsilon K^2 (e^{(T_2 - \tau)/K} - 1)] \nrm[L^2(Q_\tau)]{G}^2 + \epsilon K (e^{(T_2 - \tau)/K} - 1) \mathcal E_\tau(u).
	\end{align*}
	By setting $\epsilon = [2K(e^{(T_2 - \tau)/K} - 1)]^{-1}$ and absorbing $\mathcal E_\tau(u)$ on the right-hand side (RHS) by the left-hand side (LHS), we obtain
	\begin{align}
	\frac 1 2 \mathcal E_\tau(u)
	& \leq e(\tau) + \frac 3 2 \int_{\Sigma_\tau} |u_t u_{\nu,A}| \dif \sigma + \frac 1 2 [2K (e^{(T_2 - \tau)/K} - 1) + \frac K 2] \nrm[L^2(Q_\tau)]{G}^2 \nonumber \\
	& \leq e(\tau) + \frac 3 2 \int_{\Sigma_\tau} |u_t u_{\nu,A}| \dif \sigma + \frac {T_2 - \tau} 2 [2\frac {K e^{(T_2 - \tau)/K}} {T_2 - \tau} - \frac 3 2 \frac K {T_2 - \tau}] \nrm[L^2(Q_\tau)]{G}^2. \label{eq:0g031-HR21}
	\end{align}
	When we treat the coefficient in front of $\nrm[L^2(Q_\tau)]{G}^2$ as a function of $K$, then elementary calculus shows the minimum value of the coefficient is in between $3.5$ and $4$, i.e.,
	\begin{equation*}
		3.5 < \min_{K > 0} \{2\frac {K e^{(T_2 - \tau)/K}} {T_2 - \tau} - \frac 3 2 \frac K {T_2 - \tau} \} < 4, \ \ \text{provided} \ \ \tau < T_2.
	\end{equation*}
	Hence, by choosing the value of $K$ according to $T_2$ and $\tau$ properly, \eqref{eq:0g031-HR21} can be improved to
	\begin{equation*}
	\mathcal E_\tau(u)
	\leq 2e(\tau) + 3 \nrm[L^2(\Sigma_\tau)]{f_t} \nrm[L^2(\Sigma_\tau)]{u_{\nu,A}} + 4 (T_2 - \tau) \nrm[L^2(Q_\tau)]{G}^2,
	\end{equation*}
	which is the conclusion.
\end{proof}

\begin{Remark} \label{rem:pie-HR21}
	The piecewise smoothness of $S$ does not cause any trouble to the proof of Lemma \ref{lem:Bdd-HR21}.
	This is because the integration operation does not require the integrand to be smooth everywhere.
\end{Remark}

The RHS of the inequality in Lemma \ref{lem:Bdd-HR21} involves a norm of $u_{\nu,A}$,
and to achieve an {\it a-prior} estimate of $u$ on $\Gamma_S$, we also need to estimate the term $\nrm[L^2(\Sigma_0)]{u_{\nu,A}}$.
This can be done by playing with \eqref{eq:uu3-HR21}.
%The multiplier ``$\vec \varphi \cdot \nabla u$'' in \eqref{eq:uu3-HR21} is a generalization of `$x \cdot \nabla u$'' which has been seen in the literature, e.g.~\cite[(12)]{mor1961decay}, \cite[Lemma 2.1]{ike2005var}, \cite[Lemma 3.3]{RakeshSalo1}.

\begin{Lemma} \label{lem:tdlv-HR21}
	Under the same condition as in Proposition \ref{prop:Bddd-HR21},
%	we further assume $0 \leq \tau \leq T_1$.
	for vector-valued function $\vec \varphi \in C^1(\overline \Omega, \Rn)$ satisfying $\vec{\varphi}(x) = A(x) \nu_\Sigma(x)$ on $\partial \Omega$, we have
	\begin{align*}
	\nrm[L^2(\Sigma_\tau)]{u_{\nu,A}}^2
	& \leq  C \nrm[H^1(\Sigma_\tau)]{f}^2 + C (\sup_{\Omega} |\vec \varphi|) \big[ \nrm[L^2(\Omega)]{\nabla u(\cdot, \tau)}^2 + \nrm[L^2(\Omega)]{u_t(\cdot, \tau)}^2 \big] \nonumber \\
	& \quad + C [T_2 - \tau + (\sup_{\Omega} |\vec \varphi|)] \big[ (T_2 - \tau) \nrm[L^2(Q_\tau)]{G}^2 + \mathcal E_\tau(u) \big] \nonumber \\
	& \quad + C [T_2 - \tau + (\sup_{\Omega} |\vec \varphi|)]^2 \nrm[L^2(\Sigma_\tau)]{f_t}^2,
	\end{align*}
	for some constant $C$ depending only on $A$ and $n$.
\end{Lemma}

\begin{proof}
	\begin{align*}
		2 \Re \{\overline{u}_t [\partial_t^2 u - & \nabla \cdot (A(x) \nabla u) + qu] \} \nonumber \\
		= & \, \Re \divr_{x,t} \big[ -2 \overline{u}_t A(x) \nabla u, |u_t|^2 + |\nabla u|_A^2 + q(x)|u|^2 \big], \\
		2 \Re \{ (\vec \varphi \cdot \nabla \overline u) & [\partial_t^2 u - \nabla \cdot (A(x) \nabla u) + qu] \} \nonumber \\
		= & \, \Re \divr_{x,t} \big[ \vec \varphi (|\nabla u|_A^2 - |u_t|^2 + q|u|^2) - 2(\vec \varphi \cdot \nabla \overline u) A\nabla u, 2 (\vec \varphi \cdot \nabla \overline u) u_t \big] \nonumber \\
		& - ( \nabla \cdot \vec \varphi) (|\nabla u|_A^2 - |u_t|^2 + q|u|^2) - (\vec \varphi \cdot \nabla A)(\nabla u, \nabla u) \nonumber \\
		& + 2 (\partial_j \varphi_k) \Re (\overline u_k a_{jl} u_l) - (\nabla q \cdot \vec \varphi) |u|^2,
	\end{align*}
	Integrating \eqref{eq:uu3-HR21} in $Q_\tau$, we have
	\begin{align}
	& 2 \Re \int_{Q_\tau} (\vec \varphi \cdot \nabla \overline u) G
	= 2 \Re \int_{Q_\tau} (\vec \varphi \cdot \nabla \overline u) [\partial_t^2 u - \nabla \cdot (A(x) \nabla u) + qu] \nonumber \\
	= & \Re \int_{Q_\tau} \divr_{x,t} \big[ \vec \varphi (|\nabla u|_A^2 - |u_t|^2 + q|u|^2) - 2(\vec \varphi \cdot \nabla \overline u) A\nabla u, 2 (\vec \varphi \cdot \nabla \overline u) u_t \big] \nonumber \\
	& + \int_{Q_\tau} [ (\nabla \cdot \vec \varphi) (|u_t|^2 - |\nabla u|_A^2 - q|u|^2) - (\vec \varphi \cdot \nabla A)(\nabla u, \nabla u)] \nonumber \\
	& + \int_{Q_\tau} [2 (\partial_j \varphi_k) \Re (\overline u_k a_{jl} u_l) - (\nabla q \cdot \vec \varphi) |u|^2] \nonumber \\
	= & \Re \int_{\Sigma_\tau} \big[ (\nu_\Sigma \cdot \vec \varphi) (|\nabla u|_A^2 - |u_t|^2 + q|u|^2) - 2(\vec \varphi \cdot \nabla \overline u) u_{\nu,A} \big] \nonumber \\
	& + \Re \int_{\Gamma_{S,\tau}} \big[ (\nu_x \cdot \vec \varphi) (|\nabla u|_A^2 - |u_t|^2 + q|u|^2) + 2(\vec \varphi \cdot \nabla \overline u) (\nu_t u_t - \nu_x \cdot A\nabla u) \big] \nonumber \\
	& - \Re \int_{\Omega \times \{\tau\}} 2 (\vec \varphi \cdot \nabla \overline u) u_t \nonumber \\
	& + \int_{Q_\tau} [ (\nabla \cdot \vec \varphi) (|u_t|^2 - |\nabla u|_A^2 - q|u|^2) - (\vec \varphi \cdot \nabla A)(\nabla u, \nabla u)] \nonumber \\
	& + \int_{Q_\tau} [2 (\partial_j \varphi_k) \Re (\overline u_k a_{jl} u_l) - (\nabla q \cdot \vec \varphi) |u|^2] \nonumber \\
	=: & \, I_1 + I_2 + I_3 - \Re \int_{\Omega \times \{\tau\}} 2 (\vec \varphi \cdot \nabla \overline u) u_t, \label{eq:igI-HR21}
	\end{align}
	where $I_1$ and $I_2$ represent the integrals on $\Sigma_\tau$ and $\Gamma_{S,\tau}$, respectively, and $I_3$ represents the two integrals $Q_\tau$.
	We estimate $I_1$, $I_2$ and $I_3$ separately.

	Recall that $\vec \varphi(x) = A(x) \nu_\Sigma(x)$ on $\partial \Omega$,
	so $I_1$ can be simplified as
	\begin{equation} \label{eq:igI11-HR21}
		I_1 = \int_{\Sigma_\tau} \big[ |\nu_\Sigma|_A^2 (|\nabla u|_A^2 - |u_t|^2 + q|u|^2) - 2|u_{\nu,A}|^2 \big] \dif \sigma.
	\end{equation}
	By Lemma \ref{lem:nd-HR21} we can obtain
	\begin{equation} \label{eq:igI12-HR21}
	|\nu_\Sigma|_A^2 |\nabla u|_A^2
	\leq \frac 3 2 |u_{\nu,A}|^2 +  C |\nabla_\Sigma u|^2 \ \text{on} \ \Sigma,
	\end{equation}
	for some constant $C$ depending only on $n$ and $c_1$, $c_2$ in \eqref{eq:ellp-HR21}, where $\nabla_\Sigma u$ represents the tangential gradient of $u$ on $\Sigma$.
	Substituting \eqref{eq:igI12-HR21} into \eqref{eq:igI11-HR21}, we can continue
	\begin{align}
	I_1
	& \leq \int_{\Sigma_\tau} \big[ (\frac 3 2 |u_{\nu,A}|^2 + C |\nabla_\Sigma u|^2 + |\nu_\Sigma|_A^2 |u_t|^2 + |\nu_\Sigma|_A^2 q|u|^2) - 2|u_{\nu,A}|^2 \big] \dif \sigma \nonumber \\
	& \leq -\frac 1 2 \nrm[L^2(\Sigma_\tau)]{u_{\nu,A}}^2 + C \nrm[H^1(\Sigma_\tau)]{u}^2, \label{eq:igI1-HR21}
	\end{align}
	for some constant $C$ depending on $c_1$, $c_2$, $n$.

	The integral $I_2$ is given on $\Gamma_S$, and on $\Gamma_S$ we have
	\[
	\nabla u(x,S(x)) = \nabla \big( u(x,S(x)) \big) - \nabla S(x) u_t(x,S(x)).
	\]
	For simplicity we abbreviate $\nabla \big( u(x,S(x)) \big)$ as $\nabla \big( u \big)$.
	Then we compute
	\begin{align}
	& \, (\nu_x \cdot \vec \varphi) (|\nabla u|_A^2 - |u_t|^2 + q|u|^2) \nonumber \\
	= & \, (\nu_x \cdot \vec \varphi) \big[ |\nabla \big( u \big) - \nabla S u_t|_A^2 - |u_t|^2 + q|u|^2 \big] \nonumber \\
	= & \, (\nu_x \cdot \vec \varphi) \big[ |\nabla \big( u \big)|_A^2 - 2\Re \{ \nabla S \overline u_t \cdot A \nabla \big( u \big) \} - (1 - |\nabla S|_A^2) |u_t|^2 + q|u|^2 \big], \label{eq:igI21-HR21}
	\end{align}
	and
	\begin{align}
	& \, 2(\vec \varphi \cdot \nabla \overline u) (\nu_t u_t - \nu_x \cdot A\nabla u) \times (1+|\nabla S|^2)^{1/2} \nonumber \\
	= & \, 2[\vec \varphi \cdot \nabla \big( \overline u \big) - \vec \varphi \cdot \nabla S \overline u_t] [\nu_t u_t - \nu_x \cdot A \nabla \big( u \big) + \nu_x \cdot A \nabla S u_t] \times (1+|\nabla S|^2)^{1/2} \nonumber \\
	= & \, 2[\vec \varphi \cdot \nabla \big( \overline u \big) - \vec \varphi \cdot \nabla S \overline u_t] [u_t + \nabla S \cdot A \nabla \big( u \big) - |\nabla S|_A^2 u_t] \qquad (\text{recall \eqref{eq:nu-HR21}}) \nonumber \\
%	= & \, 2[\vec \varphi \cdot \nabla \big( \overline u \big) - \vec \varphi \cdot \nabla S \overline u_t] [(1 - |\nabla S|_A^2) u_t + \nabla S \cdot A \nabla \big( u \big)] \nonumber \\
	= & \, 2\vec \varphi \cdot \nabla \big( \overline u \big) (1 - |\nabla S|_A^2) u_t - 2 (\vec \varphi \cdot \nabla S) (1 - |\nabla S|_A^2) |u_t|^2 \nonumber \\
	& \, + 2\vec \varphi \cdot \nabla \big( \overline u \big) \nabla S \cdot A \nabla \big( u \big) - 2(\vec \varphi \cdot \nabla S) \overline u_t \nabla S \cdot A \nabla \big( u \big). \label{eq:igI22-HR21}
	\end{align}
	Combining \eqref{eq:igI21-HR21} with \eqref{eq:igI22-HR21}, we obtain
	\begin{align}
	& \Re \big[ (\nu_x \cdot \vec \varphi) (|\nabla u|_A^2 - |u_t|^2 + q|u|^2) + 2(\vec \varphi \cdot \nabla \overline u) (\nu_t u_t - \nu_x \cdot A\nabla u) \big] \times (1+|\nabla S|^2)^{1/2} \nonumber \\
	%
%	= & \, -(\nabla S \cdot \vec \varphi) |\nabla \big( u \big)|_A^2 \nonumber \\
%	+ & \, 2(\nabla S \cdot \vec \varphi) \Re \{ \nabla S \overline u_t \cdot A \nabla \big( u \big) \} \nonumber \\
%	+ & \, (\nabla S \cdot \vec \varphi) (1 - |\nabla S|_A^2) |u_t|^2 - (\nabla S \cdot \vec \varphi) q|u|^2 \nonumber \\
%	+ & \, \Re 2\vec \varphi \cdot \nabla \big( \overline u \big) (1 - |\nabla S|_A^2) u_t - 2 (\vec \varphi \cdot \nabla S) (1 - |\nabla S|_A^2) |u_t|^2 \nonumber \\
%	& \, + 2 \Re \vec \varphi \cdot \nabla \big( \overline u \big) \nabla S \cdot A \nabla \big( u \big) - 2(\vec \varphi \cdot \nabla S) \Re \{\overline u_t \nabla S \cdot A \nabla \big( u \big)\} \nonumber \\
	%
	= & \, -(\nabla S \cdot \vec \varphi) \big[ |\nabla \big( u \big)|_A^2 + (1 - |\nabla S|_A^2) |u_t|^2 + q|u|^2 \big] \nonumber \\
	& \, + 2(\nabla S \cdot \vec \varphi) \Re \{ \nabla S \overline u_t \cdot A \nabla \big( u \big) \} \nonumber \\
	& \, + 2\Re\{ [\vec \varphi \cdot \nabla \big( \overline u \big)] (1 - |\nabla S|_A^2) u_t + \nabla S \cdot A \nabla \big( u \big) \} \nonumber \\
	& \, - 2(\vec \varphi \cdot \nabla S) \Re \{\overline u_t \nabla S \cdot A \nabla \big( u \big)\} \nonumber \\
	= & -(\vec \varphi \cdot \nabla S) \big[ |\nabla \big( u \big)|_A^2 + (1 - |\nabla S|_A^2) |u_t|^2 + q|u|^2 \big] \nonumber \\
	& + 2\Re \{ [\vec \varphi \cdot \nabla \big( \overline u \big)] [\nabla S \cdot A \nabla \big( u \big) + (1 - |\nabla S|_A^2) u_t] \}. \label{eq:igI23-HR21}
	\end{align}
	With the help of \eqref{eq:igI23-HR21}, we can estimate $I_2$ in the following way,
	\begin{align}
	|I_2|
	& = \Big| \Re \int_{\Gamma_{S,\tau}} \big[ (\nu_x \cdot \vec \varphi) (|\nabla u|_A^2 - |u_t|^2 + q|u|^2) + 2(\vec \varphi \cdot \nabla \overline u) (\nu_t u_t - \nu_x \cdot A\nabla u) \big] \dif \sigma \Big| \nonumber \\
	& \leq \int_{\Gamma_{S,\tau}} |\vec \varphi \cdot \nabla S| \big[ |\nabla \big( u \big)|_A^2 + (1 - |\nabla S|_A^2) |u_t|^2 + q|u|^2 \big] \frac {\dif \sigma} {\sqrt{1 + |\nabla S|^2}} \nonumber \\
	& \quad + 2 \int_{\Gamma_{S,\tau}} |\vec \varphi \cdot \nabla \big( \overline u \big)| \big[ |\nabla S \cdot A \nabla \big( u \big)| + (1 - |\nabla S|_A^2) |u_t| \big] \frac {\dif \sigma} {\sqrt{1 + |\nabla S|^2}} \nonumber \\
	& \leq (c_1^{-1/2} + 2) (\sup_{\Omega} |\vec \varphi|) \int_{\Gamma_{S,\tau}} \big[ |\nabla \big( u \big)|_A^2 + (1 - |\nabla S|_A^2) |u_t|^2 + q|u|^2 \big] \frac {\dif \sigma} {\sqrt{1 + |\nabla S|^2}} \nonumber \\
%	& \leq (c_1^{-1} + 2) (\sup_{\Omega} |\vec \varphi|) \int_{\Omega} \big[ |\nabla \big( u \big)|_A^2 + (1 - |\nabla S|_A^2) |u_t|^2 \big] \dif x \nonumber \\
	& = 2(c_1^{-1/2} + 2) (\sup_{\Omega} |\vec \varphi|) \mathcal E_\tau(u). \label{eq:igI2-HR21}
	\end{align}

	For $I_3$, we have
	\begin{align*}
	|I_3|
	& \leq \int_{Q_\tau} |(\nabla \cdot \vec \varphi) (|u_t|^2 - |\nabla u|_A^2 - q|u|^2) - (\vec \varphi \cdot \nabla A)(\nabla u, \nabla u)| \nonumber \\
	& \quad + \int_{Q_\tau} |2 (\partial_j \varphi_k) \Re (\overline u_k a_{jl} u_l) - (\nabla q \cdot \vec \varphi) |u|^2| \nonumber \\
%	& \leq \int_{Q_\tau} |(\nabla \cdot \vec \varphi) (|u_t|^2 - |\nabla u|_A^2) - (\vec \varphi \cdot \nabla A)(\nabla u, \nabla u) + 2 (\partial_j \varphi_k) \Re (\overline u_k a_{jl} u_l)| \nonumber \\
	& \leq C \int_{Q_\tau}  (|\nabla u|_A^2 + |u_t|^2 + q|u|^2) \dif x \dif t
	= C \int_\tau^{T_2} e(s) \dif s,
	\end{align*}
	for some constant $C$ depending only on $q$, $\vec \varphi$, $A$ and $n$, and the $e(s)$ is defined in \eqref{eq:0gh-HR21}.
	Hence by \eqref{eq:0g12-HR21} which requires Assumption \ref{asp:nS-HR21}, we can have
	\begin{align}
	|I_3|
	& \leq C \int_\tau^{T_2} e^{(T_2 - s)/K} \big[ K \nrm[L^2(Q_s)]{G}^2 + \mathcal E_s(u) + 2 \int_{\Sigma_s}  |u_t u_{\nu,A}| \dif \sigma \big] \dif s \nonumber \\
	& \leq C \int_\tau^{T_2} e^{(T_2 - s)/K} \dif s \cdot \big[ K \nrm[L^2(Q_\tau)]{G}^2 + \mathcal E_\tau(u) + 2 \int_{\Sigma_\tau}  |u_t u_{\nu,A}| \dif \sigma \big] \nonumber \\
	& \leq CK (e^{(T_2 - \tau)/K} - 1) \big[ K \nrm[L^2(Q_\tau)]{G}^2 + \mathcal E_\tau(u) + \frac 1 \epsilon \nrm[L^2(\Sigma_\tau)]{u_t}^2 + \epsilon \nrm[L^2(\Sigma_\tau)]{u_{\nu, A}}^2 \big]. \label{eq:igI3-HR21}
	\end{align}
	In \eqref{eq:igI-HR21}, there is a term $\Re \int_{Q_\tau} (\vec \varphi \cdot \nabla \overline u) G$, and similar to the estimation of $I_3$, we can also estimate this integral as follows,
	\begin{align}
	& \, |\Re \int_{Q_\tau} (\vec \varphi \cdot \nabla \overline u) G|
	\leq \frac 1 2 (\sup_{\Omega} |\vec \varphi|) \int_{Q_\tau} (c_1^{-1} K^{-1} |\nabla u|_A^2 + K |G|^2) \nonumber \\
	\leq & \, \frac 1 {c_1} (\sup_{\Omega} |\vec \varphi|) \frac 1 K \int_\tau^{T_2} e(s) \dif s + \frac 1 2 (\sup_{\Omega} |\vec \varphi|) K \nrm[L^2(Q_\tau)]{G}^2 \nonumber \\
	\leq & \, \frac 1 {c_1} (\sup_{\Omega} |\vec \varphi|) (e^{(T_2 - \tau)/K} - 1) \big[ K \nrm[L^2(Q_\tau)]{G}^2 + \mathcal E_\tau(u) + \frac 1 \epsilon \nrm[L^2(\Sigma_\tau)]{u_t}^2 + \epsilon \nrm[L^2(\Sigma_\tau)]{u_{\nu, A}}^2 \big] \nonumber \\
	& +  \frac 1 2 (\sup_{\Omega} |\vec \varphi|) K \nrm[L^2(Q_\tau)]{G}^2 \nonumber \\
	\leq & \, (\sup_{\Omega} |\vec \varphi|) \cdot K (e^{(T_2 - \tau)/K} - 1 + \frac {c_1} 2) \cdot \nrm[L^2(Q_\tau)]{G}^2 \nonumber \\
	& + \frac 1 {c_1} (\sup_{\Omega} |\vec \varphi|) (e^{(T_2 - \tau)/K} - 1) \big[ \mathcal E_\tau(u) + \frac 1 \epsilon \nrm[L^2(\Sigma_\tau)]{u_t}^2 + \epsilon \nrm[L^2(\Sigma_\tau)]{u_{\nu, A}}^2 \big]. \label{eq:igI4-HR21}
	\end{align}
	Combining \eqref{eq:igI3-HR21} and \eqref{eq:igI4-HR21} and setting $K = T_2 - \tau$, we have
	\begin{align}
	& |I_3| + |\Re \int_{Q_\tau} (\vec \varphi \cdot \nabla \overline u) G| \nonumber \\
	\leq & C [T_2 - \tau + (\sup_{\Omega} |\vec \varphi|)] \big[ (T_2 - \tau) \nrm[L^2(Q_\tau)]{G}^2 + \mathcal E_\tau(u) + \frac 1 \epsilon \nrm[L^2(\Sigma_\tau)]{u_t}^2 + \epsilon \nrm[L^2(\Sigma_\tau)]{u_{\nu, A}}^2 \big], \label{eq:igI34-HR21}
	\end{align}
	for some constant $C$ depending only on $\sup_\Omega |\vec \varphi|$, $A$ and $n$.
	It seems the constant $C$ also rely on $\sup_\Omega |\vec \varphi|$.
	However, because the only requirement on $\vec \varphi$ is ``$\vec \varphi(x) = A(x) \nu_\Sigma(x)$ on $\partial \Omega$'', we can choose $\vec \varphi$ such that $\sup_\Omega |\vec \varphi| \leq 2 \sup_{\partial \Omega} |\vec \varphi|$.
	By doing this, we can have
	\[
	\sup_\Omega |\vec \varphi|
	\leq 2 \sup_{\partial \Omega} |\vec \varphi|
	= 2 \sup_{\partial \Omega} |A(x) \nu_\Sigma(x)|
	\leq 2 \sup_{\partial \Omega} |A|
	\leq 2 \sup_{\Rn} |A|.
	\]
	Therefore, it is enough only claim $C$ depends on $A$ and $n$.
	This also implies the results in this work is also valid when $\Omega$ is unbounded or when $\Omega = \Rn$.
	In what follows we will not emphasize this unless otherwise is necessary.
	
	Now, by combining \eqref{eq:igI1-HR21}, \eqref{eq:igI2-HR21} and \eqref{eq:igI34-HR21} with \eqref{eq:igI-HR21}, we obtain
	\begin{align}
	& \frac 1 2 \nrm[L^2(\Sigma_\tau)]{u_{\nu,A}}^2 \nonumber \\
	\leq & \, C \nrm[H^1(\Sigma_\tau)]{u}^2 + |I_2| + 2 |\int_{Q_\tau} (\vec \varphi \cdot \nabla \overline u) G| + |I_3| + |\int_{\Omega \times \{\tau\}} 2 (\vec \varphi \cdot \nabla \overline u) u_t| \nonumber \\
	\leq & \, C \nrm[H^1(\Sigma_\tau)]{u}^2 + 2(c_1^{-1/2} + 2) (\sup_{\Omega} |\vec \varphi|) \mathcal E_\tau(u) + (\sup_{\Omega} |\vec \varphi|) \big[ \nrm[L^2(\Omega)]{\nabla u(\cdot, \tau)}^2 + \nrm[L^2(\Omega)]{u_t(\cdot, \tau)}^2 \big] \nonumber \\
	& + C [T_2 - \tau + (\sup_{\Omega} |\vec \varphi|)] \big[ (T_2 - \tau) \nrm[L^2(Q_\tau)]{G}^2 + \mathcal E_\tau(u) + \frac 1 \epsilon \nrm[L^2(\Sigma_\tau)]{u_t}^2 + \epsilon \nrm[L^2(\Sigma_\tau)]{u_{\nu, A}}^2 \big]. \nonumber
	\end{align}
	By setting $\epsilon = \{ 4C [T_2 - \tau + (\sup_{\Omega} |\vec \varphi|)] \}^{-1}$, we arrive at
	\begin{align*}
	\frac 1 4 \nrm[L^2(\Sigma_\tau)]{u_{\nu,A}}^2
	& \leq  C \nrm[H^1(\Sigma_\tau)]{u}^2 + C (\sup_{\Omega} |\vec \varphi|) \big[ \nrm[L^2(\Omega)]{\nabla u(\cdot, \tau)}^2 + \nrm[L^2(\Omega)]{u_t(\cdot, \tau)}^2 \big] \nonumber \\
	& \quad + C [T_2 - \tau + (\sup_{\Omega} |\vec \varphi|)] \big[ (T_2 - \tau) \nrm[L^2(Q_\tau)]{G}^2 + \mathcal E_\tau(u) \big] \nonumber \\
	& \quad + C [T_2 - \tau + (\sup_{\Omega} |\vec \varphi|)]^2 \nrm[L^2(\Sigma_\tau)]{u_t}^2.
	\end{align*}
	The proof is done.
\end{proof}

With the help of Lemmas \ref{lem:Bdd-HR21} and \ref{lem:tdlv-HR21}, we are able to bound $\mathcal E(u)$ and $\nrm[L^2(\Sigma_0)]{u_{\nu,A}}$ by the initial/boundary data and the source term.

\begin{Lemma} \label{lem:Bddd-HR21}
	Under the same condition as in Proposition \ref{prop:Bddd-HR21}, we have
	\begin{align*}
	\mathcal E(u)
	& \leq C \mathcal E(u,0) + C \agl[T_2] (\nrm[L^2(Q_0)]{G}^2 + \nrm[H^1(\Sigma_0)]{f}^2), \\
	\nrm[L^2(\Sigma_0)]{u_{\nu,A}}^2
	& \leq C \agl[T_2] \mathcal E(u,0) + C \agl[T_2]^2 (\nrm[H^1(\Sigma_0)]{f}^2 + \nrm[L^2(Q_0)]{G}^2),
	\end{align*}
	for some constant $C$ depending only on $A$, $q$ and the dimension $n$.
\end{Lemma}

\begin{proof}
	Using the inequality in Lemma \ref{lem:Bdd-HR21} with $\tau = 0$, and noting that $u(\cdot,0) = u_0$ and $u_t(\cdot,0) = u_1$, we have
	\begin{equation} \label{eq:Huf1-HR21}
	\mathcal E(u)
	\leq \mathcal E(u,0) + \frac 9 {4\epsilon} \nrm[L^2(\Sigma_0)]{f_t}^2 + \epsilon \nrm[L^2(\Sigma_0)]{u_{\nu,A}}^2 + 4 T_2 \nrm[L^2(Q_0)]{G}^2.
	\end{equation}
	Substituting the inequality in Lemma \ref{lem:tdlv-HR21} with $\tau = 0$ into \eqref{eq:Huf1-HR21},
	and setting $\epsilon$ in \eqref{eq:Huf1-HR21} to be $\{2C[T_2 + (\sup_{\Omega} |\vec \varphi|)]\}^{-1}$, we obtain
	\begin{equation} \label{eq:Huf-HR21}
	\mathcal E(u)
	\leq C \mathcal E(u,0) + C [T_2 + (\sup_{\Omega} |\vec \varphi|) + 1] (\nrm[H^1(\Sigma_0)]{f}^2 + \nrm[L^2(Q_0)]{G}^2),
	\end{equation}
	for some constant $C$ depending only on $q$, $\vec \varphi$, $A$ and $n$.
	Recall that in Lemma \ref{lem:tdlv-HR21} we have fixed the value of $\vec \varphi$ on $\partial \Omega$.
	But we still have the freedom to choose the value of $\vec \varphi$ in the interior of $\Omega$.
	We can choose $\vec \varphi$ in such a way that $\sup_\Omega |\vec \varphi| \leq 2 \sup_{\partial \Omega} |\vec \varphi|$, and this choice guarantees
	\[
	\sup_\Omega |\vec \varphi| \leq 2 \sup_{\partial \Omega} |A \cdot \nu_\Sigma| \leq 2 c_2,
	\]
	for the constant $c_2$ given in \eqref{eq:ellp-HR21}.
	Combining this with \eqref{eq:Huf-HR21}, we arrive at the first inequality of the lemma.

	For the second inequality, we substitute the inequality in Lemma \ref{lem:Bdd-HR21} into the inequality in Lemma \ref{lem:tdlv-HR21} with $\tau = 0$, and this gives
	\begin{align*}
		\nrm[L^2(\Sigma_0)]{u_{\nu,A}}^2
		& \leq  C [T_2 + (\sup_{\Omega} |\vec \varphi|)] \mathcal E(u,0) \nonumber \\
		& \quad + C [T_2 + (\sup_{\Omega} |\vec \varphi|)]^2 \nrm[H^1(\Sigma_0)]{f}^2 + C [T_2 + (\sup_{\Omega} |\vec \varphi|)] T_2 \nrm[L^2(Q_0)]{G}^2 \nonumber \\
		& \quad + C [T_2 + (\sup_{\Omega} |\vec \varphi|)] [ \frac 1 \epsilon \nrm[L^2(\Sigma_0)]{f_t}^2 + \epsilon \nrm[L^2(\Sigma_0)]{u_{\nu,A}}^2 + T_2 \nrm[L^2(Q_0)]{G}^2 ].
	\end{align*}
	By letting $\epsilon = \{2C [T_2 + (\sup_{\Omega} |\vec \varphi|)]\}^{-1}$ and absorbing the $\nrm[L^2(\Sigma_0)]{u_{\nu,A}}^2$-term on the RHS by the LHS, we obtain
	\begin{align*}
	\nrm[L^2(\Sigma_0)]{u_{\nu,A}}^2
	& \leq  C [T_2 + (\sup_{\Omega} |\vec \varphi|)] \mathcal E(u,0) + C [T_2 + (\sup_{\Omega} |\vec \varphi|)]^2 (\nrm[H^1(\Sigma_0)]{f}^2 + \nrm[L^2(Q_0)]{G}^2).
	\end{align*}
	Again, by using $\sup_\Omega |\vec \varphi| \leq 2 c_2$, we obtain the second inequality.
	The proof is complete.
\end{proof}

\begin{Remark} \label{rem:unA-HR21}
	When the hypersurface $\Gamma_S$ is horizontal, the corresponding estimate of the Neumann data $u_{\nu,A}$ given in Lemma \ref{lem:Bddd-HR21} will be a quantitative version of \cite[eq.~(2.7)]{llt1986non}.
\end{Remark}

Now we are ready to prove Proposition \ref{prop:Bddd-HR21}.

\begin{proof}[Proof of Proposition \ref{prop:Bddd-HR21}]
	In \eqref{eq:0g0-HR21}, when $\tau = 0$, we have
	\begin{align}
	|\mathcal E(u) - \mathcal E(u,0)|
	& = |\Re \int_{\Sigma_0} \overline{f}_t u_{\nu,A} \dif \sigma + \Re \int_{Q_0} \overline{u}_t G| \nonumber \\
%	& \leq \int_{\Sigma_0} |f_t| |u_{\nu,A}| \dif \sigma + \int_{Q_0} |u_t| |G| \nonumber \\
	& \leq \nrm[L^2(\Sigma_0)]{f_t} \nrm[L^2(\Sigma_0)]{u_{\nu,A}} + \nrm[L^2(Q_0)]{u_t} \nrm[L^2(Q_0)]{G}. \label{eq:eu01-HR21}
	\end{align}
	For $\nrm[L^2(Q_0)]{u_t}$, noticing that $\nrm[L^2(H_\tau)]{u_t}^2 \leq e(\tau)$, so by \eqref{eq:0g12-HR21} with $K$ set to be $T_2$, we have
	\begin{align*}
	\frac 1 2 \nrm[L^2(Q_0)]{u_t}^2
	& = \int_0^{T_2} \frac 1 2 \nrm[L^2(H_\tau)]{u_t}^2 \dif \tau
	\leq \int_0^{T_2} e(\tau) \dif \tau \\
	& \leq T_2 (e^{T_2/T_2} - 1) (\mathcal E(u) + \frac {T_2} 2 \nrm[L^2(Q_0)]{G}^2 + \int_{\Sigma_0}  |u_t u_{\nu,A}| \dif \sigma) \\
	& \leq C T_2 \big( \mathcal E(u) + \frac {T_2} 2 \nrm[L^2(Q_0)]{G}^2 + \nrm[L^2(\Sigma_0)]{f_t} \nrm[L^2(\Sigma_0)]{u_{\nu,A}} \big).
	\end{align*}
	Combining this with the estimates of $\mathcal E(u)$ and $\nrm[L^2(\Sigma_0)]{u_{\nu,A}}$ given in Lemma \ref{lem:Bddd-HR21}, we obtain
	\begin{align}
	\nrm[L^2(Q_0)]{u_t}^2
	& \leq C T_2 \mathcal E(u,0) + C (T_2+1)^2 \big( \nrm[L^2(Q_0)]{G}^2 + \nrm[H^1(\Sigma_0)]{f}^2 \big) \nonumber \\
	& \quad + C T_2^2 \nrm[L^2(Q_0)]{G}^2 + C T_2^2 \nrm[L^2(\Sigma_0)]{f_t}^2 + C (T_2 + 1) \mathcal E(u,0) \nonumber \\
	& \quad + C (T_2 + 1)^2 (\nrm[H^1(\Sigma_0)]{f}^2 + \nrm[L^2(Q_0)]{G}^2) \nonumber \\
	& \leq C \agl[T_2] \mathcal E(u,0) + C \agl[T_2]^2 \big( \nrm[L^2(Q_0)]{G}^2 + \nrm[H^1(\Sigma_0)]{f}^2 \big), \label{eq:eu02-HR21}
	\end{align}
	for some constant $C$ depending only on $A$ and the dimension $n$.
	Substituting \eqref{eq:eu02-HR21} into \eqref{eq:eu01-HR21}, we obtain
	\begin{align*}
	|\mathcal E(u) - \mathcal E(u,0)|
	& \leq \nrm[L^2(\Sigma_0)]{f_t} \nrm[L^2(\Sigma_0)]{u_{\nu,A}} + C \agl[T_2]^{1/2} \sqrt{\mathcal E(u,0)} \nrm[L^2(Q_0)]{G} \nonumber \\
	& \quad + C\agl[T_2] (\nrm[L^2(Q_0)]{G} + \nrm[H^1(\Sigma_0)]{f}) \nrm[L^2(Q_0)]{G}.
	\end{align*}
	Again, by the estimate of $\nrm[L^2(\Sigma_0)]{u_{\nu,A}}$ given in Lemma \ref{lem:Bddd-HR21}, we arrive at
	\begin{align*}
	|\mathcal E(u) - \mathcal E(u,0)|
	& \leq C \agl[T_2]^{1/2} \big[ \sqrt{\mathcal E(u,0)} + \agl[T_2]^{1/2} (\nrm[H^1(\Sigma_0)]{f} + \nrm[L^2(Q_0)]{G}) \big] \nonumber \\
	& \quad \times (\nrm[H^1(\Sigma_0)]{f} + \nrm[L^2(Q_0)]{G})
	\end{align*}
	The proof is complete.
\end{proof}

\section{Approximation of the solution} \label{sec:app-HR21}

The results in Section \ref{sec:est-HR21} are based on the prerequisite $u \in C^2$.
This is not true for the system \eqref{eq:1-HR21} when:
\begin{enumerate}
	\item compatibility issue: only the compatibility condition up to order zero is satisfied;
	
	\item regularity issue: the initial/boundary data are merely $H^1$.
\end{enumerate}
However, these two issues can all be overcome by approximation.
That is to say, we can find smooth sequences $u_{0,\epsilon}$, $u_{1,\epsilon}$, $f_\epsilon$, $G_\epsilon$ which converge to $u_0$, $u_1$, $f$, $G$, respectively,
and under these approximate data, we obtain approximate solutions $u_\epsilon$, which will converge in $H^1(\Gamma_S)$.
For the compatibility conditions issue, we shall modify $f_{\epsilon}$.
And for the regularity issue, with the help of the estimate given in Lemma \ref{lem:Bddd-HR21}, we are able to approximate the system \eqref{eq:1-HR21} with smooth enough initial/boundary data and to show the corresponding approximate solution has a limit in $H^1(\Gamma_S)$ and the corresponding energy is well-defined.

\subsection{Perturbation of the potential}

As mentioned before, when the potential $q$ is zero, the classical regularity result of $u$ in horizontal hyperplanes $\Omega \times \{t = \tau\}$ has been given in \cite{llt1986non}
For readers convenience we reproduce \cite[Remark 2.10]{llt1986non} below.

\begin{Lemma} \label{lem:LLT210-HR21}
	Assume $A$ is a $C^1$-smooth real-valued symmetric matrix function, and $\partial \Omega$ is $C^2$.
	Let $\tilde u(x,t)$ be a solution of the system
	\begin{equation*}
		\left\{\begin{aligned}
			\partial_t^2 \tilde u - \nabla \cdot (A(x) \nabla \tilde u) & = G && \text{in} \ Q, \\
			\tilde u & = f && \text{on} \ \Sigma, \\
			\tilde u(\cdot,0) = u_0,\,
			\partial_t \tilde u & = u_1 && \text{on} \ \Omega \times \{t = 0\},
		\end{aligned}\right.
	\end{equation*}
	with $(G, f, u_0, u_1)$ satisfying the regularity assumptions ($m$ is a non-negative integer)
	\begin{equation*}
		\left\{\begin{aligned}
		& G \in L^1(0,T; H^m(\Omega)), \ \partial_t^m G \in L^1(0,T; L^2(\Omega)), \\
		& f \in H^{m+1}(\Sigma) := L^2(0,T; H^{m+1}(\partial \Omega)) \cap H^{m+1}(0,T; L^2(\partial \Omega)), \\
		& u_0 \in H^{m+1}(\Omega), \ u_1 \in H^{m}(\Omega),
		\end{aligned}\right.
	\end{equation*}
	and satisfying all necessary compatibility conditions up to order $m$.
	Then
	\begin{equation*}
		\partial_k^k u_0 \in C([0,T]; H^{m+1-k}(\Omega)) \ \ \text{for} \ \ 0 \leq k \leq  m+1, \quad \text{and} \ \
		\partial_\nu u_0 \in H^m(\Sigma).
	\end{equation*}
\end{Lemma}

%\begin{Remark} \label{rem:LTEx-HR21}
%	\sq{The existence of the solution of $u_0$ is guaranteed by a combination of the harmonic extension of the lateral boundary datum $f$ with the classical Galerkin method with initial data and a source term \cite[\S 7.2]{EvanPDEs}.
%	Check \cite[Section 3]{llt1986non} for more details.}
%\end{Remark}

\begin{Remark} \label{rem:LLT210-HR21}
	For the case of Lemma \ref{lem:LLT210-HR21} where $\nabla \cdot (A(x) \nabla)$ is replaced by $\Delta$, the corresponding result is covered by \cite[Remark 2.10]{llt1986non}.
	Moreover, in \cite[Section 4]{llt1986non}, the authors discussed how to generalize from $\Delta$ to \mbox{$\nabla \cdot (A(x) \nabla)$} when $m = 0,1$.
	Actually, by following the same steps in the proof of \cite[Theorem 2.2]{llt1986non}, we can generalize the scenario to any integer $m$ not only $0$ and $1$, and the proof is straightforward so we omit it.
\end{Remark}

\begin{Remark}
	Readers should note that, in this work, the function space for the source $G$ is set to be $L^2(Q)$, which is a subset of the space $L^1(0,T; H^m(\Omega))$ with $m = 0$ used in \cite[Remark 2.10]{llt1986non}.
\end{Remark}

In our case \eqref{eq:1-HR21}-\eqref{eq:1con-HR21}, however, we have a zero order perturbation $q$, so Lemma \ref{lem:LLT210-HR21} cannot directly apply.
But under certain smoothness condition of $q$,
we shall show here that $q$ does not affect the regularity result of the solution.
To that end, we divide \eqref{eq:1-HR21}-\eqref{eq:1con-HR21} into the following two PDEs for $v$ and $w$ respectively, and we see $u = v + w$,
\begin{equation} \label{eq:utv-HR21}
	\left\{\begin{aligned}
		\partial_t^2 v - \nabla \cdot (A(x) \nabla v) & = G && \text{in} \ Q, \\
		v & = f && \text{on} \ \Sigma, \\
		v = u_0,\,
		\partial_t v & = u_1 && \text{on} \ \Omega \times \{t = 0\},
	\end{aligned}\right.
\end{equation}
and
\begin{equation} \label{eq:utw-HR21}
	\left\{\begin{aligned}
		\partial_t^2 w - \nabla \cdot (A(x) \nabla w) + qw & = -q(x) v(x,t) && \text{in} \ Q, \\
		w & = 0 && \text{on} \ \Sigma, \\
		w = 0,\,
		\partial_t w & = 0 && \text{on} \ \Omega \times \{t = 0\}.
	\end{aligned}\right.
\end{equation}
Applying Lemma \ref{lem:LLT210-HR21} to \eqref{eq:utv-HR21} gives
\[
\partial_t^k v \in C([0,T]; H^{m+1-k}(\Omega)) \ \ \text{for} \ \ 0 \leq k \leq  m+1, \quad \text{and} \ \
\partial_\nu v \in H^m(\Sigma).
\]
Because $q \in C_c^{m+1}(\Omega)$, we have
\[
\partial_t^k (qv) \in C([0,T]; H^{m+1-k}(\Omega)) \ \ \text{for} \ \ 0 \leq k \leq  m+1, \quad \text{and} \ \
\partial_\nu (qv) = 0 \ \text{on} \ \Sigma
\]
as well.
$qv$ is the source term of \eqref{eq:utw-HR21}, so by \cite[\S 7.2.3 Theorem 6]{EvanPDEs} to $w$, we see
\[
\partial_t^k w \in L^\infty(0,T; H^{m+1-k}(\Omega)) \ \ \text{for} \ \ 0 \leq k \leq  m+1, \quad \text{and} \ \ \partial_\nu w \in H^{m-1}(\Sigma).
\]
then by applying \cite[\S 7.2.3 Theorem 6]{EvanPDEs} the time-$L^\infty$ can be transformed into the $C^k$-norm with the total regularity drop by one,
\[
\partial_t^k w \in C([0,T]; H^{m-k}(\Omega)) \ \ \text{for} \ \ 0 \leq k \leq  m, \quad \text{and} \ \
\partial_\nu w \in H^{m-1}(\Sigma).
\]
Summing up $v+w$, we obtain
\[
\partial_t^k u \in C([0,T]; H^{m-k}(\Omega)) \ \ \text{for} \ \ 0 \leq k \leq  m, \quad \text{and} \ \
\partial_\nu u \in H^{m-1}(\Sigma).
\]
We have proved the following claim.

\begin{Proposition} \label{prop:LLTg-HR21}
	Under the same assumptions as in Lemma \ref{lem:LLT210-HR21}, but with an additional term $q\tilde u$ in the equation, with $q \in C_c^{m+1}(\overline \Omega)$, we have
	\begin{equation*}
		\partial_k^k u_0 \in C([0,T]; H^{m-k}(\Omega)) \ \ \text{for} \ \ 0 \leq k \leq  m, \quad \text{and} \ \
		\partial_\nu u_0 \in H^{m-1}(\Sigma).
	\end{equation*}
\end{Proposition}

\subsection{The compatibility issue}

We can remove the smoothness condition ``$u \in C^2(\overline \Omega)$'' in Lemma \ref{lem:Bddd-HR21} which is implicitly stipulated due to Proposition \ref{prop:Bddd-HR21}.

\begin{Lemma} \label{lem:1ep-HR21}
	Given Assumption \ref{asp:nS-HR21}.
	In system \eqref{eq:1-HR21}, assume $q \in C_c^{\lceil n/2 \rceil + 3}(\overline \Omega)$, $f \in C^\infty(\overline \Sigma)$, $G \in C^\infty(\overline Q)$, $u_0$, $u_1 \in C^\infty(\overline \Omega)$ with
	\(
	u_0(x) = f(x,0) \ \text{on} \ \partial \Omega.
	\)
	Then we have $u \in H^1(\Gamma_S)$ and
	\begin{align*}
		|\mathcal E(u) - \mathcal E(u,0)|
		& \leq C \agl[T_2]^{1/2} \big[ \sqrt{\mathcal E(u,0)} + \agl[T_2]^{1/2} (\nrm[H^1(\Sigma_0)]{f} + \nrm[L^2(Q_0)]{G}) \big] \nonumber \\
		& \quad \times (\nrm[H^1(\Sigma_0)]{f} + \nrm[L^2(Q_0)]{G}),
	\end{align*}
	and
	\begin{equation*}
		\nrm[L^2(\Sigma_0)]{u_{\nu,A}}^2
		\leq C \agl[T_2] \mathcal E(u,0) + C \agl[T_2]^2 (\nrm[L^2(Q_0)]{G}^2 + \nrm[H^1(\Sigma_0)]{f}^2).
	\end{equation*}
\end{Lemma}

\begin{proof}
	We shall find a sequence of $C^2$-smooth solutions $\{ u_\epsilon \}_{\epsilon > 0}$ of
	\begin{equation} \label{eq:1ep-HR21}
		\left\{\begin{aligned}
		\partial_t^2 u_\epsilon - \nabla \cdot (A(x) \nabla u_\epsilon) + qu_\epsilon & = G && \text{in} \ Q, \\
		u_\epsilon & = f_\epsilon && \text{on} \ \Sigma, \\
		u_\epsilon = u_0,\,
		\partial_t u_\epsilon & = u_1 && \text{on} \ \Omega \times \{t = 0\},
		\end{aligned}\right.
	\end{equation}
	by using the regularity result in Proposition \ref{prop:LLTg-HR21}.
	For this, we need smoothness of $u_0$, $u_1$, $f$ and $G$, which are already assumed in this lemma,
	and we also need certain higher order compatibility conditions to be satisfied on $\Omega \times \{t = 0\}$.
	To guarantee the compatibility conditions, we modify the Dirichlet boundary datum $f$.
	
	Let us construct a series of Dirichlet boundary data $\{ f_\epsilon \}_{\epsilon > 0}$ in the following way.
	First, we define $u_k$ in $\overline \Omega$ iteratively for $2 \leq k \leq \frac {K-1} 2$ using $u_0$ and $u_1$,
	\begin{equation*}
		u_k(x) := \partial_t^{k - 2} G(x,0) +\nabla \cdot (A(x) \nabla u_{k - 2}(x)), \quad x \in \overline \Omega.
	\end{equation*}
	The value of the integer $K$ shall be determined later.
	Recall that $u_0$, $u_1 \in C^\infty(\overline \Omega)$,
	and also recall the smoothness of $A$ stipulated at the beginning of the article, i.e.~$A \in C^{n+4}(\overline \Omega)$.
	This guarantees
	\begin{equation} \label{eq:uks-HR21}
		u_{2k},\, u_{2k+1} \in C^{n+5-2k}(\overline \Omega), \quad \forall 1 \leq k \leq \frac {K-1} {2}.
	\end{equation}
	Then, we fix a cutoff function $\chi \in C_c^\infty(\R)$ satisfying $\chi(t) = 1$ when $|t| \leq 1$ and $\chi(t) = 0$ when $|t| \geq 2$,
	and we set
	\begin{equation} \label{eq:uep-HR21}
		f_\epsilon(x,t)
		:= \chi(t/\epsilon) \sum_{k = 0}^{K} \frac {t^k} {k!} u_k(x) + (1 - \chi(t/\epsilon)) f(x,t), \quad \forall (x,t) \in \overline \Sigma.
	\end{equation}
	By \eqref{eq:uks-HR21} and \eqref{eq:uep-HR21}, $f_\epsilon \in C^{n+5-K}(\overline \Sigma)$.
	Now compatibility conditions up to order $K$ required by Proposition \ref{prop:LLTg-HR21}, i.e.,
	\begin{equation*}
		u_k = \partial_t^k f_\epsilon \text{~on~} \partial \Omega \times \{ t = 0 \}, \ \ \forall k:  0 \leq k \leq K,
	\end{equation*}
	are all satisfied, see also \cite[\S 7.2.3 eq.~(62)]{EvanPDEs}.
	Recall the condition $q \in C_c^{\lceil n/2 \rceil + 3}(\overline \Omega)$.
	When $\lceil n/2 \rceil + 2 \geq N := \min\{n+5-K, K\}$, we can use Proposition \ref{prop:LLTg-HR21} up to order $N$ to conclude that the corresponding solution $u_\epsilon$ satisfies
	\[
	u_\epsilon \in C([0,T]; H^N(\Omega)), \quad
	\partial_t u_\epsilon \in C([0,T]; H^{N-1}(\Omega)), \quad
	\partial_t^2 u_\epsilon \in C([0,T]; H^{N-2}(\Omega)).
	\]
	By the Sobolev embedding theorem we know $H^{N-2+j}(\Omega) \hookrightarrow C^{j}(\overline{\Omega})$ for $j=0,1,2$ when $N - 2 \geq n/2$, so we set $K = \lceil n/2 \rceil + 2$, and thus
	\begin{equation*}
		u_\epsilon \in C([0,T]; C^2(\overline{\Omega})), \quad
		\partial_t u_\epsilon \in C([0,T]; C^1(\overline{\Omega})), \quad
		\partial_t^2 u_\epsilon \in C([0,T]; C(\Omega)),
	\end{equation*}
	which implies
	\begin{equation} \label{eq:uec-HR21}
		u_\epsilon \in C^2(\overline Q).
	\end{equation}
	
	The $C^2$-smoothness of $u_\epsilon$ guarantees us to use Proposition \ref{prop:Bddd-HR21} to conclude
	\begin{align}
		|\mathcal E(u_\epsilon) & - \mathcal E(u,0)| \nonumber \\
		& \leq C \agl[T_2]^{1/2} \big[ \sqrt{\mathcal E(u,0)} + \agl[T_2]^{1/2} (\nrm[H^1(\Sigma_0)]{f_\epsilon} + \nrm[L^2(Q_0)]{G}) \big] \nonumber \\
		& \quad \times (\nrm[H^1(\Sigma_0)]{f_\epsilon} + \nrm[L^2(Q_0)]{G}), \label{eq:mHu-HR21}
	\end{align}
	for certain constant $C$ depending only on $A$ and $n$. Further, by the linearity of \eqref{eq:1ep-HR21} we obtain
	\begin{equation} \label{eq:cH12-HR21}
		\mathcal E(u_{\epsilon_1} - u_{\epsilon_2})
		\leq C \agl[T_2] \nrm[H^1(\Sigma_0)]{f_{\epsilon_1} - f_{\epsilon_2}}^2, \quad \forall \epsilon_1, \epsilon_2 > 0.
	\end{equation}

	Besides compatibility conditions, the construction \eqref{eq:uep-HR21} also guarantees $f_\epsilon \to f$ in $H^1(\Sigma)$.
	Indeed, it is straightforward to check that
	\begin{equation} \label{eq:ue1-HR21}
		\nrm[L^2(\Sigma)]{f - f_\epsilon} \to 0, \quad
		\nrm[L^2(\Sigma)]{\nabla|_{\Sigma,x}(f - f_\epsilon)} \to 0, \quad \text{as~} \epsilon \to 0^+,
	\end{equation}
	where $\nabla|_{\Sigma,x}$ stands for the spatial tangential gradient on $\Sigma$.
	Therefore, to show $f_\epsilon \to f$ in $H^1(\Sigma)$, it is left to show $\nrm[L^2(\Sigma)]{\partial_t(f - f_\epsilon)} \to 0$ as $\epsilon \to 0^+$.
	By \eqref{eq:uep-HR21} one can compute
	\begin{align*}
		\partial_t (f_\epsilon - f)(x,t)
		& = \frac 1 \epsilon \chi'(\frac t \epsilon) [\sum_{k = 0}^{K} \frac {t^k} {k!} u_k(x) - f(x,t)] +  \chi(\frac t \epsilon) [\sum_{k = 0}^{K-1} \frac {t^k} {k!} u_k(x) - f'(x,t)] \\
		& = \frac 1 \epsilon \chi'(\frac t \epsilon) [\sum_{k = 0}^{K} \frac {t^k} {k!} u_k(x) - u_0(x) - f'(x,\xi_t) t] +  \chi(\frac t \epsilon) [\sum_{k = 0}^{K-1} \frac {t^k} {k!} u_k - f'] \\
		& = \frac t \epsilon \chi'(\frac t \epsilon) [\sum_{k = 1}^{K} \frac {t^{k-1}} {k!} u_k(x) - f'(x,\xi_t)] +  \chi(\frac t \epsilon) [\sum_{k = 0}^{K-1} \frac {t^k} {k!} u_k(x) - f'(x,t)],
	\end{align*}
	where we used the compatibility condition
	``\(
	u_0(x) = f(x,0) \ \text{on} \ \partial \Omega
	\)''
	and Taylor's expansion with Lagrange remainder, and the $\xi_t$ comes from the Lagrange remainder.
	Note that for every $\epsilon > 0$, $\frac {(\cdot)} \epsilon \chi'(\frac \cdot \epsilon)$ and $\chi(\frac \cdot \epsilon)$ are uniformly bounded by $\max\{2 \sup|\chi'| ,1\}$ in $[0,T]$,
	so $|\partial_t (f_\epsilon - f)|^2$ is dominated by an integrable function $F$ given as follows,
	\[
	F(x,t) := \max\{2 \sup|\chi'| ,1\} ^2 \big( |\sum_{k = 1}^{K} \frac {t^{k-1}} {k!} u_k(x) - f'(x,\xi_t)| + |\sum_{k = 0}^{K-1} \frac {t^k} {k!} u_k(x) - f'(x,t)| \big)^2.
	\]
	Also, from \eqref{eq:uep-HR21} we see $f_\epsilon = f$ when $t \geq 2\epsilon$, so $\partial_t (f_\epsilon - f) = 0$ when $t \geq 2\epsilon$ and hence $\partial_t (f_\epsilon - f) \to 0$ almost everywhere in $\Sigma$ as $\epsilon \to 0^+$.
	Therefore, by Lebesgue's dominated convergence theorem, we obtain
	\begin{equation} \label{eq:ue2-HR21}
		\nrm[L^2(\Sigma)]{\partial_t (f - f_\epsilon)} \to 0, \quad \text{as~} \epsilon \to 0^+.
	\end{equation}
	Combining \eqref{eq:ue2-HR21} with \eqref{eq:ue1-HR21}, we arrive at
	\begin{equation} \label{eq:ue-HR21}
		\nrm[H^1(\Sigma)]{f - f_\epsilon} \to 0, \quad \text{as~} \epsilon \to 0^+.
	\end{equation}
	
	Now, combining \eqref{eq:ue-HR21} with \eqref{eq:cH12-HR21}, we see $\mathcal E(u_{\epsilon_1} - u_{\epsilon_2})$ goes to zero as $\epsilon_1, \epsilon_2 \to 0^+$,
	which implies $\{ \nabla_{\Gamma_S} u_\epsilon \}$ is a Cauchy sequence in $L^2(\Gamma_S)$ and $\{ \partial_t u_\epsilon \}$ is a Cauchy sequence in $L^2(\Gamma_S)$.
	We claim that ``$\{ \nabla_{\Gamma_S} u_\epsilon \}$ being a Cauchy sequence in $L^2(\Gamma_S)$'' is enough to conclude $\{ u_\epsilon \}$ is a Cauchy sequence in $H^1(\Gamma_S)$.
	This is due to the reason that every $u_\epsilon$ has the same trace on $\partial \Gamma_S$.
	Hence, we can define $u|_{\Gamma_S}$ as the limit of $u_\epsilon|_{\Gamma_S}$, i.e.,
	\[
	u|_{\Gamma_S} := \lim_{\epsilon \to 0^+} u_\epsilon |_{\Gamma_S},
	\]
	and the estimate \eqref{eq:cH12-HR21} implies the limit $u|_{\Gamma_S}$ is in $H^1({\Gamma_S})$.
	Moreover, the energy $\mathcal E(u)$ is well-defined and by Lemma \ref{lem:tri-HR21}  and we see $\mathcal E(u_\epsilon) \to \mathcal E(u)$ as $\epsilon \to 0^+$.
	Now, taking the limit of \eqref{eq:mHu-HR21}, we obtain
	\begin{align*}
		|\mathcal E(u) & - \mathcal E(u,0)| \nonumber \\
		& \leq C \agl[T_2]^{1/2} \big[ \sqrt{\mathcal E(u,0)} + \agl[T_2]^{1/2} (\nrm[H^1(\Sigma_0)]{f} + \nrm[L^2(Q_0)]{G}) \big] \nonumber \\
		& \quad \times (\nrm[H^1(\Sigma_0)]{f} + \nrm[L^2(Q_0)]{G}).
	\end{align*}
	
	For $\nrm[L^2(\Sigma_0)]{u_{\nu,A}}$, by combining arguments above with Lemma \ref{lem:Bddd-HR21}, we can see $u_{\nu,A}$ is also well-define on $\Sigma_0$ and
	\begin{equation*}
		\nrm[L^2(\Sigma_0)]{u_{\nu,A}}^2
		\leq C \agl[T_2] \mathcal E(u,0) + C \agl[T_2]^{2} \nrm[L^2(Q_0)]{G}^2 + \nrm[H^1(\Sigma_0)]{f}^2.
	\end{equation*}
	The proof is complete.
\end{proof}

\subsection{The regularity issue}

The $C^\infty$-smooth regularity requirements for $u_0$, $u_1$, $f$ and $G$ in Lemma \ref{lem:1ep-HR21} can be further improved.

\begin{Lemma} \label{lem:2ep-HR21}
	Under the assumptions as in Lemma \ref{lem:1ep-HR21}, but with the regularity conditions of $f$, $G$, $u_0$, $u_1$ replaced by \eqref{eq:1con-HR21}, there hold $u \in H^1(\Gamma_S)$ and
	\begin{align*}
		|\mathcal E(u) - \mathcal E(u,0)|
		& \leq C \agl[T_2]^{1/2} \big[ \sqrt{\mathcal E(u,0)} + \agl[T_2]^{1/2} (\nrm[H^1(\Sigma_0)]{f} + \nrm[L^2(Q_0)]{G}) \big] \nonumber \\
		& \quad \times (\nrm[H^1(\Sigma_0)]{f} + \nrm[L^2(Q_0)]{G}).
	\end{align*}
\end{Lemma}

\begin{proof}
	One can find four sequences $\{ u_{0,\epsilon} \}_{\epsilon > 0}$, $\{ u_{1,\epsilon} \}_{\epsilon > 0}$, $\{ \tilde f_{\epsilon} \}_{\epsilon > 0}$, $\{ G_{\epsilon} \}_{\epsilon > 0}$, satisfying the following requirements:
	\begin{equation} \label{eq:einS-HR21}
		\left\{\begin{aligned}
			\{ u_{0,\epsilon} \}_{\epsilon > 0} \subset C^\infty(\overline \Omega) & \text{~such that~} \nrm[H^1(\Omega)]{u_{0, \epsilon} - u_0} \leq \epsilon, \\
			\{ u_{1,\epsilon} \}_{\epsilon > 0} \subset C^\infty(\overline \Omega) & \text{~such that~} u_{1, \epsilon} \to u_1 \text{~in~} L^2(\Omega), \\
			\{ \tilde f_{\epsilon} \}_{\epsilon > 0} \subset C^\infty(\overline \Sigma) & \text{~such that~} \nrm[H^1(\Sigma)]{\tilde f_{\epsilon} - f} \leq \epsilon, \\
			\{ G_{\epsilon} \}_{\epsilon > 0} \subset C^\infty(\overline Q) & \text{~such that~} G_{\epsilon} \to G \text{~in~} L^2(Q).
		\end{aligned}\right.
	\end{equation}
	The condition ``$u_0(x) = f(x,0)$ on $\partial \Omega$'' mentioned in \eqref{eq:1con-HR21} does not guarantee $u_{0,\epsilon}(x) = \tilde f_\epsilon(x,0)$ on $\partial \Omega$.
	Hence we need to modify $\tilde f_\epsilon$ to $f_\epsilon$ so that $(u_{0,\epsilon}, u_{1,\epsilon}, f_\epsilon, G_\epsilon)$ meet the requirements of Lemma \ref{lem:1ep-HR21}.
	
	Similar to \eqref{eq:uep-HR21}, we fix a cutoff function $\chi \in C_c^\infty(\R)$ satisfying $\chi(t) = 1$ when $|t| \leq 1$ and $\chi(t) = 0$ when $|t| \geq 2$,
	and set
	\begin{equation} \label{eq:uep2-HR21}
		f_\epsilon(x,t)
		:= \chi(t/\epsilon) u_{0, \epsilon}(x) + (1 - \chi(t/\epsilon)) \tilde f_\epsilon(x,t), \quad \forall (x,t) \in \overline \Sigma.
	\end{equation}
	Then,
	\begin{equation} \label{eq:ufc-HR21}
		f_\epsilon \in C^\infty(\overline \Omega), \quad \text{and} \quad
		u_{0,\epsilon}(x) = f_\epsilon(x,0) \ \text{on} \ \partial \Omega.
	\end{equation}
	It can be seen that (see \eqref{eq:ue1-HR21})
	\begin{equation} \label{eq:uf1-HR21}
		\nrm[L^2(\Sigma)]{f - f_\epsilon} \to 0, \quad
		\nrm[L^2(\Sigma)]{\nabla|_{\Sigma,x}(f - f_\epsilon)} \to 0, \quad \text{as~} \epsilon \to 0^+,
	\end{equation}
	so, to guarantee $f_\epsilon \to f$ in $H^1(\Sigma)$, it is left to show $\nrm[L^2(\Sigma)]{\partial_t(f - f_\epsilon)} \to 0$ as $\epsilon \to 0^+$.
	By \eqref{eq:uep2-HR21} we have
	\begin{align}
		& \, \partial_t (f_\epsilon - f)(x,t)
		= \partial_t(\tilde f_\epsilon - f) - \chi(t/\epsilon) \tilde f_\epsilon' + \frac 1 \epsilon \chi'(t/\epsilon) (u_{0, \epsilon}(x) - \tilde f_\epsilon(x,t)) \nonumber \\
		= & \, \partial_t(\tilde f_\epsilon - f) - \chi(t/\epsilon) \tilde f_\epsilon' \nonumber \\
		& + \frac t \epsilon \chi'(t/\epsilon) \frac {[u_{0, \epsilon}(x) - u_0(x)] + [f(x,0) - \tilde f_\epsilon(x,0)] + [\tilde f_\epsilon(x,0)- \tilde f_\epsilon(x,t)]} t, \label{eq:fef-HR21}
	\end{align}
	where we used the compatibility condition $u_0(x) = f(x,0)$ on $\partial \Omega$.

	The function $\frac {(\cdot)} \epsilon \chi'(\frac {\cdot} \epsilon)$ is supported in the interval $[\epsilon, 2\epsilon]$ and is bounded by $2 \max|\chi'|$, and $\nrm[H^1(\Omega)]{u_{0, \epsilon} - u_0} \leq \epsilon$, so
	\begin{align}
		\int_\Sigma |\frac t \epsilon \chi'(\frac t \epsilon) \frac {u_{0, \epsilon}(x) - u_0(x)} t|^2
		& \leq C \int_\epsilon^{2\epsilon} \big( \int_{\partial \Omega} |\frac {u_{0, \epsilon}(x) - u_0(x)} \epsilon|^2 \big) \dif t \nonumber \\
		& = C \epsilon^{-1} \nrm[L^2(\partial \Omega)]{u_{0, \epsilon} - u_0}^2
		\leq C \epsilon^{-1} \nrm[H^1(\Omega)]{u_{0, \epsilon} - u_0}^2 \nonumber \\
		& \leq C \epsilon, \label{eq:fef1-HR21}
	\end{align}
	where we used the trace theorem.
	Similarly, we have
	\begin{align}
		\int_\Sigma |\frac t \epsilon \chi'(\frac t \epsilon) \frac {f(x,0) - \tilde f_\epsilon(x,0)} t|^2
		& \leq C \epsilon^{-1} \nrm[L^2(\partial \Omega)]{f - \tilde f_\epsilon}^2
		\leq C \epsilon^{-1} \nrm[H^1(\Sigma)]{f - \tilde f_\epsilon}^2
		\leq C \epsilon. \label{eq:fef2-HR21}
	\end{align}
	We can also estimate the last term in \eqref{eq:fef-HR21}.
	Note that $\tilde f_{\epsilon}$ is a smooth function, so we have
	\begin{align*}
		\int_\Sigma |\frac t \epsilon \chi'(\frac t \epsilon) \frac {\tilde f_\epsilon(x,0)- \tilde f_\epsilon(x,t)} t|^2 \dif \sigma
		& = \int_\Sigma |\frac t \epsilon \chi'(\frac t \epsilon)|^2 |\partial_t \tilde f_\epsilon(x,t) + \mathcal O(t)|^2 \dif \sigma \nonumber \\
		& \leq 2 \int_\Sigma |\frac t \epsilon \chi'(\frac t \epsilon)|^2 |\partial_t \tilde f_\epsilon(x,t)|^2 \dif \sigma + \mathcal O(\epsilon^2).
	\end{align*}
	The function $|\frac t \epsilon \chi'(\frac t \epsilon)|^2 |\partial_t \tilde f_\epsilon(x,t)|^2$ is dominated by $C |\partial_t \tilde f_\epsilon(x,t)|^2$ for certain constant $C$,
	whose Lebesgue integral in $\Sigma$ is bounded by $C \nrm[H^1(\Sigma)]{\tilde f_\epsilon}$, and hence bounded by $C \nrm[H^1(\Sigma)]{f}+1$ when $\epsilon$ is small enough.
	Also, the function $|\frac t \epsilon \chi'(\frac t \epsilon)|^2 |\partial_t \tilde f_\epsilon(x,t)|^2$ converges to $0$ for $\forall (x,t) \in \overline{\Sigma}$ as $\epsilon \to 0^+$,
	which is because $|\frac t \epsilon \chi'(\frac t \epsilon)|^2$ converges to $0$ for $\forall t \in [0,T]$ as $\epsilon \to 0^+$.
	Therefore, by Lebesgue's dominated convergence theorem we have
	\begin{equation} \label{eq:fef3-HR21}
		\lim_{\epsilon \to 0^+} \int_\Sigma |\frac t \epsilon \chi'(\frac t \epsilon) \frac {\tilde f_\epsilon(x,0)- \tilde f_\epsilon(x,t)} t|^2 \dif \sigma
		\leq 2 \lim_{\epsilon \to 0^+} \int_\Sigma |\frac t \epsilon \chi'(\frac t \epsilon)|^2 |\partial_t \tilde f_\epsilon(x,t)|^2 \dif \sigma = 0.
	\end{equation}
	Combining \eqref{eq:fef1-HR21}, \eqref{eq:fef2-HR21}, \eqref{eq:fef3-HR21} with \eqref{eq:fef-HR21}, we conclude $\nrm[L^2(\Sigma)]{\partial_t(f_\epsilon - f)} \to 0$, so
	\begin{equation} \label{eq:fe-HR21}
		\nrm[H^1(\Sigma)]{f_\epsilon - f} \to 0, \quad \epsilon \to 0^+.
	\end{equation}

	By \eqref{eq:einS-HR21} and \eqref{eq:ufc-HR21},
	we see $(u_{0,\epsilon}, u_{1,\epsilon}, f_\epsilon, G_\epsilon)$ meet the requirement of Lemma \ref{lem:1ep-HR21}, so the corresponding solution $u_\epsilon$ satisfies
	\begin{align}
		|\mathcal E(u_\epsilon) - \mathcal E(u_\epsilon,0)|
		& \leq C \agl[T_2]^{1/2} \big[ \sqrt{\mathcal E(u_\epsilon,0)} + \agl[T_2]^{1/2} (\nrm[H^1(\Sigma_0)]{f_\epsilon} + \nrm[L^2(Q_0)]{G_\epsilon}) \big] \nonumber \\
		& \quad \times (\nrm[H^1(\Sigma_0)]{f_\epsilon} + \nrm[L^2(Q_0)]{G_\epsilon}), \label{eq:fel2-HR21}
	\end{align}
	and
	\begin{align}
		& \, |\mathcal E(u_{\epsilon_1} - u_{\epsilon_2}) - \mathcal E(u_{\epsilon_1} - u_{\epsilon_2},0)| \nonumber \\
		\leq & \, C \agl[T_2]^{1/2} \big[ \sqrt{\mathcal E(u_{\epsilon_1} - u_{\epsilon_2},0)} + \agl[T_2]^{1/2} (\nrm[H^1(\Sigma_0)]{f_{\epsilon_1} - f_{\epsilon_2}}  \nonumber \\
		& \, + \nrm[L^2(Q_0)]{G_{\epsilon_1} - G_{\epsilon_2}}) \big]
		\times (\nrm[H^1(\Sigma_0)]{f_{\epsilon_1} - f_{\epsilon_2}} + \nrm[L^2(Q_0)]{G_{\epsilon_1} - G_{\epsilon_2}}). \label{eq:fel1-HR21}
	\end{align}
	Here $\mathcal E(u_\epsilon, 0)$ and $\mathcal E(u_{\epsilon_1} - u_{\epsilon_2},0)$ signifies
	\begin{align*}
	\mathcal E(u_\epsilon, 0)
	& = \nrm[L^2(\Omega)]{\nabla u_{0,\epsilon}}^2 + \nrm[L^2(\Omega)]{u_{1,\epsilon}}^2, \\
	\mathcal E(u_{\epsilon_1} - u_{\epsilon_2},0)
	& = \nrm[L^2(\Omega)]{\nabla (u_{0,\epsilon_1} - u_{0,\epsilon_2})}^2 + \nrm[L^2(\Omega)]{u_{1,\epsilon_1} - u_{1,\epsilon_2}}^2,
	\end{align*}
	respectively.
	Combining \eqref{eq:fel1-HR21} with these limits given in \eqref{eq:einS-HR21} and \eqref{eq:fe-HR21}, we see the limit of $\{u_\epsilon\}$ and $\{\partial_t u_\epsilon\}$ exist in $H^1(\Gamma_S)$ and in $L^2(\Sigma_S)$, respectively, and the limit coincides with the solution $u$.
	Therefore, by taking the limit of \eqref{eq:fel2-HR21}, we arrive at the conclusion.
\end{proof}

\section{Proofs of the main results} \label{sec:sqrt-HR21}

We are ready to prove Theorem \ref{thm:1H1-HR21}.

\begin{proof}[Proof of Theorem \ref{thm:1H1-HR21}]
	We introduce an auxiliary function $v$ which satisfies
	\begin{equation} \label{eq:v-HR21}
		\left\{\begin{aligned}
			\partial_t^2 v - \nabla \cdot (A(x) \nabla v) + qv & = 0 && \text{in} \ Q, \\
			u & = 0 && \text{on} \ \Sigma, \\
			v = u_0,\,
			v & = u_1 && \text{on} \ \Omega \times \{t = 0\},
		\end{aligned}\right.
	\end{equation}
	By Lemma \ref{lem:2ep-HR21}, we see $\mathcal E(v; \Gamma_S) = \mathcal E(v,0)$.
	And $\mathcal E(v,0) = \mathcal E(u,0)$.
	From \eqref{eq:tri1-HR21} we can see
	\[
	|\sqrt{\mathcal E(u_1)} - \sqrt{\mathcal E(u_2)}|
	\leq \sqrt{\mathcal E(u_1 - u_2)}.
	\]
	Hence,
	\begin{equation*}
		|\sqrt{\mathcal E(u; \Gamma_S)} - \sqrt{\mathcal E(u,0)}|
		= |\sqrt{\mathcal E(u; \Gamma_S)} - \sqrt{\mathcal E(v,\Gamma_S)}|
		\leq \sqrt{\mathcal E(u - v; \Gamma_S)}.
	\end{equation*}
	Note that the function $u-v$ satisfies the equation \eqref{eq:1-HR21} with initial condition $u_0 = u_1 = 0$, namely $\mathcal E(u-v,0) = 0$, so Lemma \ref{lem:2ep-HR21} gives
	\[
	\sqrt{\mathcal E(u-v; \Gamma_S)}
	\leq C \agl[T_2]^{1/2} (\nrm[H^1(\Sigma_0)]{f} + \nrm[L^2(Q_0)]{G}),
	\]
	thus,
	\begin{equation*}
		|\sqrt{\mathcal E(u; \Gamma_S)} - \sqrt{\mathcal E(u,0)}|
		\leq C \agl[T_2]^{1/2} (\nrm[H^1(\Sigma_0)]{f} + \nrm[L^2(Q_0)]{G}).
	\end{equation*}
	The proof is complete.
\end{proof}

\begin{proof}[Proof of Corollary \ref{cor:1en-HR21}]
	Recall the notation $\mathcal E(u; \Gamma_\tau)$ defined before Corollary \eqref{cor:1en-HR21}.
	When $f = 0$ and $G = 0$, by Theorem \ref{thm:1H1-HR21} we have
	\begin{equation*}
		\mathcal E(u; \Gamma_\tau)
		= \frac 1 2 \int_\Omega (|\nabla u_0|_A^2 + |u_1|^2 + q|u_0|^2) \dif x,
	\end{equation*}
	which is true for any $\tau$.
\end{proof}

We can borrow arguments from the proof of Theorem \ref{thm:1H1-HR21} to show Theorem \ref{thm:2-HR21}.

\begin{proof}[Proof of Theorem \ref{thm:2-HR21}]
	The sequences $(u_{0,\epsilon}, u_{1,\epsilon}, f_\epsilon, G_\epsilon)$ constructed in the proof of Theorem \ref{thm:1H1-HR21} meet the requirement of Lemma \ref{lem:1ep-HR21}, so the corresponding solution $u_\epsilon$ satisfies
	\begin{equation*}
	\nrm[L^2(\Sigma_0)]{u_{\epsilon,\nu,A}}^2
	\leq C \agl[T_2] \mathcal E(u_\epsilon, 0) + C \agl[T_2]^{2} (\nrm[L^2(Q_0)]{G_\epsilon}^2 + \nrm[H^1(\Sigma_0)]{f_\epsilon}^2).
	\end{equation*}
	By taking the limit, we obtain the result.
\end{proof}

\subsection*{Acknowledgements}

The research of the author is partially supported by the NSF of China under the grant No.~12301540.

%\bibliography{reference_bib_HR2021}
%\bibliographystyle{alpha}

{

% \bib, bibdiv, biblist are defined by the amsrefs package.

}

\end{document}